\documentclass[preprint,11pt]{elsarticle}

\usepackage{amssymb}
\usepackage{amsmath}
\usepackage{amsthm}
\usepackage{color}

\begin{document}

\newcommand{\rc}{\color{red}}
\newcommand{\ec}{\color{black}}
\newcommand{\bc}{\color{blue}}

\def\m{{\mathfrak m}} 

\def\p{{\mathfrak p}}    
\def\q{{\mathfrak q}}    
\def\n{{\mathfrak n}} 

\def\S{{\mathcal S}} 
\def\T{{\mathcal T}} 
\def\U{{\mathcal U}} 
\def\O{{\mathcal O}} 
\def\R{{\mathcal R}} 
\def\E{{\mathcal E}} 
\def\B{{\mathcal B}}

\def\I{{\mathcal I}} 
\def\v{{\mathfrak v}}
\def\V{{\mathcal V}}
\def\W{{\mathcal w}}

\def\N{\mathbb N} 

\def\SN{\mathbb N} 
\def\SQ{{\mathbb Q}} 
\def\SZ{\mathbb Z} 
\def\SF{\mathbb F} 
\def\SR{\mathbb R} 
\def\SC{\mathbb C} 

\def\spec{\operatorname{Spec }} 

\newtheorem{theorem}{Theorem}[section]
\newtheorem{lemma}[theorem]{Lemma}
\newtheorem{proposition}[theorem]{Proposition}
\newtheorem{corollary}[theorem]{Corollary}
\newtheorem{problem}[theorem]{Problem}

\theoremstyle{definition}
\newtheorem{defi}[theorem]{Definitions}
\newtheorem{definition}[theorem]{Definition}
\newtheorem{remark}[theorem]{Remark}
\newtheorem{example}[theorem]{Example}
\newtheorem{question}[theorem]{Question}
\newtheorem{comment}[theorem]{Comment}
\newtheorem{comments}[theorem]{Comments}
\newtheorem{construction}[theorem]{Construction}
\newtheorem{notation}[theorem]{Notation}
\newtheorem{notation and remarks}[theorem]{Notation and Remarks}
\newtheorem{facts}[theorem]{Facts}
\newtheorem{fact}[theorem]{Fact}

\newtheorem{discussion}[theorem]{Discussion}
\newtheorem{setting}[theorem]{Setting}

\renewcommand{\thedefi}{}

\def\ff{\frak}
\def\Spec{\mbox{\rm Spec }}
\def\Proj{\mbox{\rm Proj }}
\def\Rees{\mbox{\rm Rees }}

\def\hgt{\mbox{\rm ht }}
\def\type{\mbox{ type}}
\def\Hom{\mbox{ Hom}}
\def\rank{\mbox{ rank}}
\def\Ext{\mbox{ Ext}}
\def\Ker{\mbox{ Ker}}
\def\Max{\mbox{\rm Max}}
\def\End{\mbox{\rm End}}
\def\xpd{\mbox{\rm xpd}}
\def\Ass{\mbox{\rm Ass}}
\def\emdim{\mbox{\rm emdim}}
\def\epd{\mbox{\rm epd}}
\def\repd{\mbox{\rm rpd}}
\def\ord{\mbox{\rm ord}}

\begin{frontmatter}



\title{The tree of quadratic transforms of a    regular local ring of dimension two}

\author[Heinzer]{William Heinzer}
\address[Heinzer]{Department of Mathematics, Purdue University, West
Lafayette, Indiana 47907-1395 U.S.A.}
\ead{heinzer@purdue.edu}

\author[Loper]{K. Alan Loper}
\address[Loper]{Department of Mathematics, Ohio State University -- Newark,  Newark, OH 43055}
\ead{lopera@math.ohio-state.edu}

\author[Olberding]{Bruce Olberding} 
\address[Olberding]{Department of Mathematical Sciences, New Mexico State University,
 Las Cruces, NM 88003-8001 U.S.A.}
 \ead{olberdin@nmsu.edu}




\begin{abstract}
Let $D$ be a 2-dimensional regular local ring and let  $Q(D)$  denote the quadratic tree of 2-dimensional 
regular local overrings of $D$.  We  explore the topology of the tree $Q(D)$ and the  
  family  $ \R(D)$  of rings   obtained as  intersections of rings in $Q(D)$. 
If $A$ is a finite intersection of rings in $ Q(D)$, then $A$ is Noetherian and the structure of $A$ is well
understood. 
 However, other rings in $\R(D)$ need not be Noetherian.
The   two main goals  of this paper are to  examine  topological  properties of  the quadratic tree  
$Q(D)$, and to   examine  the structure of  rings    in the  set $\R(D)$. 

\end{abstract}


\end{frontmatter}

  \section{Introduction}  
  
  Let $D$ be a 2-dimensional regular local ring.
   Among the  overrings of  $D$   inside the field of 
fractions of $D$,  the rings   that are 2-dimensional regular local rings form a partially ordered set  $Q(D)$  
 with respect to inclusion.  
  For rings $\alpha$ and $\beta$ in $Q(D)$ with $\alpha \subseteq \beta$,  it is
  known from work of Abhyankar,  that the regular local rings  dominating $\alpha$ 
  and dominated by $\beta$ form a finite linearly ordered chain. Thus $Q(D)$ is a tree with respect to inclusion. 
  
  The   tree   $Q(D)$  reflects  ideal-theoretic properties of  the complete ideals of $D$ 
  that are primary for the maximal ideal $\m_D$.    Zariski's theory  of 
  complete ideals implies the existence of a one-to-one correspondence between the 
  elements in each of the following 3 sets:  
  \begin{enumerate}
  \item  The 
  simple complete $\m_D$  -primary ideals.
  \item 
 The rings in the quadratic tree $Q(D)$.  
 \item
 The order valuation rings of the rings in $Q(D)$.
 \end{enumerate}

Motivation to  examine properties of overrings of $D$ that are the
intersection of elements in the quadratic tree $Q(D)$  arises from  two  sources: 
\begin{enumerate}
\item
The beautiful structure of the tree $Q(D)$.
\item
Similarities between the intersections of elements in $Q(D)$   and  the representation of a Krull domain as 
the intersection of its essential valuation rings.
\end{enumerate} 

 Let $\R(D)$ denote the family of rings obtained as intersections of rings in $Q(D)$. 

\begin{remark}  \label{disc2.7}    The Noetherian rings in $\R(D)$ are all Krull domains. 
 Associated to a Krull domain 
$A$ is a unique set of DVRs,  the set $\E(A)$  of  essential valuation rings of $A$; 
 $\E(A) =  \{A_\p \}$,  where $\p$  varies 
over  the  
height 1 prime ideals  of $A$.  Two useful properties related to $\E(A)$   are:   
\begin{enumerate}
\item
$A = \bigcap \{ A_\p~ |~ A_\p \in \E(A) \}$ and the intersection is irredundant.
\item
The set $\E(A)$ defines an essential representation of $A$.
\end{enumerate} 
Since $D$ is a Noetherian ring of Krull dimension $2$, every Krull domain between $D$ and its quotient field is a Noetherian ring \cite[Theorem 9]{HKrull}.
\end{remark}

  A goal in  this paper    
  is to examine  topological  properties of  the quadratic tree  $Q(D)$, and the structure of  rings 
  in the  set $\R(D)$. 
In  the paper   \cite{HO}   a description is given of the   Noetherian rings in $\R(D)$. 
Subsets    $\U$  of $Q(D)$ that are  closed points of nonsingular   projective models 
over $D$ are examined. The rings obtained  are Noetherian rings that  can be  described in  detail.

For example, it is shown in \cite[Corollary~6.5]{HO} that if  $n$ is  a positive integer and 
  $R$ is an irredundant intersection of $n$ 
  elements in $Q(D)$, then  $R$ is a Noetherian regular domain with precisely $n$ maximal 
  ideals, each maximal ideal of $R$ is of height 2, and the localizations of $R$ at its maximal
  ideals are the $n$ elements in $Q(D)$ that intersect irredundantly  to define  $R$.

A  focus in the current  paper is to  describe non-Noetherian rings in $\R(D)$.    
Related to  this we  examine the topological structure of the quadratic tree $Q(D)$. 
  In   Sections \ref{sec4} and \ref{sec5},  we   investigate the topology of  Noetherian 
   subsets of $Q(D)$, find  the patch limits points and prove  that  Noetherian subsets  of $Q(D)$ 
   are  precisely the  subsets that  are bounded in the
   sense that they are the points of $Q(D)$ contained in  finitely  many  {dominating}   
   valuation overrings of $D$.
  Examples in Section 6 show that for Noetherian subsets $\U$ of $Q(D)$, the structure of rings $R = \O_\U$ is more complicated  than the situation with projective models. 

Theorem~\ref{noether down} establishes  that the subsets $\U$ of $Q(D)$ that are Noetherian subspaces of
 $Q(D)$ are 
 precisely the ones for which there exist  
finitely many valuation  overrings  $V_i $  of $D$  that dominate $D$ and are 
such that each ring in $ \U $ is contained in one of the  $V_i$.  
{The closed irreducible subspaces of $Q(D)$ are all Noetherian,  and are  in one-to-one 
correspondence with the dominating valuation overrings of $D$. 
 By Lemma~\ref{irreducible}, an  irreducible component   of $Q(D)$  is either:  
 \begin{enumerate}
 \item  
 the set of rings of $Q(D)$ contained in a
prime divisor of the second kind for $D$,  or
\item 
 the set of rings of $Q(D)$  contained in a minimal valuation overring  $V$ of $D$,  
 where $V$  is not a subring of a prime divisor of the second kind on $D$.  
\end{enumerate}  }
For $\alpha \in Q(D)$,  let $P(\alpha)$ denote the set of points proximate to $\alpha$.  A finite union of sets 
of the form $ P(\alpha_i)$
 is Noetherian, and so is the set of  the closed points of a   nonsingular projective model  over $D$.
 Also,  an infinite sequence of iterated local  quadratic transforms defines a  Noetherian subspace of $Q(D)$.  
  In this paper we  are moving beyond projective models and focusing on what can be said about intersections
   of rings in Noetherian  subsets of $ Q(D)$  with special emphasis on   sets of  points 
   proximate to finitely many elements of $Q(D)$.

Theorem~\ref{thm6.5} describes an infinite  subset $\U$ of $Q(D)$ such that $\U$ defines an 
irredundant essential representation of $B = \O_\U$, and $B$ is an 
almost Krull domain\footnote{An integral domain $B$ is said to be an almost Krull domain if $B_P$  is a 
  Krull domain for each prime ideal $P$ of $B$. } that is not Noetherian.  
  Corollary~\ref{cor4.13} implies that the constructed ring  $\O_\U$  is the intersection of an almost Dedekind 
  domain\footnote{An integral domain $A$ is an almost Dedekind domain if $A_P$ is a Dedekind 
  domain for each maximal ideal $P$ of $A$.}  and a principal ideal domain.  
  
  Let $V$ be a minimal valuation overring of $D$.  
  Theorem~\ref{th1.8} establishes the existence of    a 
  subset $\U$  of $Q(D)$ such that the  ring $\O_\U = C$  
  has the property that $V$
  is a localization of $C$.  If $V$ is chosen not to be a DVR,  then $C$ is not an almost Krull domain.

Theorem~\ref{thm6.4} describes a  non-Noetherian local domain in $\R(D)$.  This is  the 
most intricate example constructed in Section~\ref{sec6}.

\section{Notation and Terminology}

We mainly follow the notation used by Matsumura in \cite{Mat}.  Thus a local 
ring need not be Noetherian. An extension ring $B$ of an integral domain $A$ 
is said to be an {\it overring of} $A$ if $B$ is a subring of the field of
fractions of $A$.

For the definition of a quadratic transform, also called a local quadratic transform, 
 we refer to   \cite[pp.~569-577]{Ab8},  
\cite[p.~367]{ZS2} and \cite[p.~263]{SH}.\footnote{What
are called quadratic transforms in \cite{Ab8} and \cite{ZS2} are called local
quadratic transforms in \cite[p.~263]{SH}.}   
The powers of the maximal ideal of a regular local ring $R$ define a 
rank one discrete valuation ring denoted $\ord_R$.   If $\dim R = d$,  then the residue
field of $\ord_R$ is a pure transcendental extension of the residue field of $R$ of 
transcendence degree $d-1$. 

Complete  ideals, also called integrally closed ideals,  are defined and studied in  the book
of Swanson and Huneke \cite{SH}. Let $R$ be a Noetherian integral domain, and let $V$, with
maximal ideal $\m_V$,  be a   
valuation overring of $R$. Let  $\p =  \m_V\cap R$.  Following notation as in \cite[Definition 9.3.1]{SH},
$V$ is said to be a {\it divisorial } valuation ring with respect 
to $R$ if the transcendence degree of $V/\m_V$ over the field $R_\p/\p R_\p$ is $\hgt \p - 1$.
Every divisorial valuation ring with respect to $R$ is Noetherian \cite[Theorem~9.3.2]{SH}. 
Divisorial valuation rings are classically called {\it prime divisors} on $R$  \cite[p.~95]{ZS2}.
$V$  is a {\it prime divisor of the first kind} if $\p = \m_V \cap R$ has height one. If $\hgt \p > 1$,
then $V$ is said to be a {\it prime divisor of the second kind}.

\begin{notation and remarks} \label{note2.1}
Let $R$ be an integral domain and let $S$ be a local overring of $R$.
\begin{enumerate}
\item 
The {\it center}  of  $S$ on $R$  is the prime ideal $\m_S \cap R$, where $\m_S$ denotes the
maximal ideal of $S$.
\item
If $R$ is  a local ring,     then $S$ is said to {\it dominate} $R$ 
if  the center of   $S$  on $R$  is  the maximal ideal of $R$, that is, $\m_S \cap R = \m_R$, 
where $\m_R$ is the  maximal ideal of  $R$.
\item
A valuation overring $V$ of $R$ is said to be a {\it minimal valuation overring} of $R$ if 
$V$ is minimal with respect to set-theoretic inclusion in the set of valuation overrrings of $R$.
\item
If $W$ is a valuation overring of $R$ and the center of $W$ on $R$ is a nonmaximal prime ideal,
then by composite construction \cite[p.~43]{Nag}, there exists a valuation overring $V$ of $R$ 
such that $V \subset W$ and  $V$ is centered on a maximal ideal of $R$.  

\item
 Assume that  $R$ is local.  Every valuation overring of $R$ contains a valuation 
 overring of $R$ that   dominates $R$.
  If $W$ is a valuation overring of $R$ that dominates $R$ and the 
 field $W/\m_W$ is not algebraic over $R/\m_R$, then by composite construction, there exists a 
 valuation overring $V$ of $R$ such that $V$ dominates $R$ and $V \subsetneq W$. 
 
 \item
 Assume that $R$ is Noetherian and local. A valuation overring $V$ of $R$ is a minimal
 valuation overring of $R$ if and only if $V$ dominates $R$ and the field $V/\m_V$ is 
 algebraic over the field $R/\m_R$.   Let $X(R)$ denote the set of minimal valuation overrings of the
 Noetherian local domain $R$. 
 
 \item
 Assume that $R$ is a regular local ring and $V$ is a prime divisor on $R$ that dominates $R$.
 Abhyankar proves in \cite[Prop.~3]{Ab1} 
  that there exists a unique  finite sequence
\begin{equation} \label{1}
R ~ =  ~R_0  ~\subset ~ R_1 ~ \subset \cdots \subset ~ R_{h}  ~\subset ~ R_{h+1} ~= ~ V
\end{equation}
of regular local rings $R_j$, where  $\dim R_h \ge 2$ and    $R_{j+1}$ is a 
 local quadratic transform  of
$R_{j}$ along $V$   for each $j \in \{0, \ldots, h \}$,
and $\ord_{R_{h} }  = V.$   The association of the prime divisor $V$ with the regular local ring $R_h$
in Equation~\ref{1}, and the uniqueness of the sequence in Equation~\ref{1} establishes a one-to-one 
correspondence between the prime divisors $V$ dominating $R$ and the regular local rings $S$ of
dimension at least 2 that dominate $R$ and are obtained from $R$ by a finite sequence of local quadratic
transforms as in Equation~\ref{1}.  The regular local rings $R_j$ with $j \le h$ in Equation~\ref{1} 
are called the {\it infinitely near points} to $R$ along $V$. 
 In general, a regular local ring $S$ of dimension at least 2  is
called an {\it infinitely near point} to $R$ {\it of level} $h$  if
there exists a sequence
$$
R ~ =  ~R_0  ~\subset ~ R_1 ~ \subset \cdots ~ \subset ~ R_{h}~ = ~ S,  \quad h ~\ge ~ 0
$$
of regular local rings $R_j$ of dimension at least 2,   where   $R_{j+1}$ is a
local quadratic transform  of
$R_{j}$   for each $j$ with $0 \le j \le h-1$
\cite[Definition~1.6]{L}.
 \item
 Assume that $R$ is a regular local ring with $\dim R = 2$, and 
 $S$ is a regular local overring of $R$.
 \begin{enumerate}
 \item  
 If  $\dim S \ge 2$, then $\dim S = 2$ and $S$ dominates $R$.
 \item
 If $\dim S = 2$ and $S \ne R$,  then $\m_R S$ is a proper principal ideal of $S$.
 It follows that  $S$ dominates a unique local quadratic transform $R_1$ of $R$.
 Moreover,  
 there exists for some positive integer $h$ a sequence
$$
R~= ~ R_0 ~\subset ~R_1 ~\subset ~ \cdots ~ \subset R_{h} ~= ~ S,
$$
where $R_j$ is a local quadratic transform  of $R_{j-1}$ for each $j \in \{1, \ldots, h\}$.
The rings $R_j$ are precisely the regular local domains  that are subrings of $S$ and contain
$R$ \cite[Theorem~3]{Ab1}. Every 2-dimensional regular local overring $S$ of $R$ is an 
infinitely near point to $R$.  Each $V \in X(R)$ is the union of the infinite quadratic sequence 
of $R$ along $V$ \cite{Ab1}.

\item  The Zariski theory about the unique factorization of  the complete ideals of 
the 2-dimensional regular local ring $R$ yields a one-to-one correspondence between the
following 3 sets: 
\begin{enumerate}
\item
 the  simple complete 
$\m_R$-primary ideals,
\item
  the  infinitely near points to $R$,  and 
 \item
  the prime divisors that 
dominate $R$.  
\end{enumerate}
See   \cite[Appendix 5]{ZS2} and \cite{Hun}.   For each infinitely near point $S$ to $R$, 
$\ord_S$ is a divisorial valuation ring with respect to $R$ that dominates $R$,  and $\ord_S$ is the 
unique Rees valuation ring of a unique simple complete $\m_R$-primary ideal.

 \end{enumerate} 
\end{enumerate} 
\end{notation and remarks}

   \section{The quadratic tree  of a 2-dimensional regular local ring}  \label{sec3}
   
   
  The following notation will be used throughout the rest of the article.

 \begin{notation} \label{note3.1} 
  Let $D$ be a $2$-dimensional regular local ring with quotient field $F$  and maximal ideal $\m_D = (x, y)D$.
  \begin{enumerate}
  \item
   Let $Q(D)$ denote the 
  set of all 2-dimensional regular local overrings of $D$.
   As    noted in Remark~\ref{note2.1}.8,  the rings in $Q(D)$ are infinitely near points to $D$. 
  We call $Q(D)$ the {\it quadratic tree} determined by $D$. 
  \item
   Regarded as a partially ordered set 
  with respect to inclusion, 
  $Q(D)$ is a tree,  and is  
  the disjoint union of subsets $Q_j(D)$, $j \ge 0$,  where $Q_0(D) = \{D\}$, and $Q_j(D)$ for
  $j \ge 1$ is the  set of infinitely near points to $D$ of level $j$ as in Remark~\ref{note2.1}.8. 
  \item
     As in  Lipman \cite{L},  it is often convenient to denote rings in $Q(D)$ 
  with lower case Greek letters.  For $\alpha \in Q(D)$,  let  $Q(\alpha)$ denote the
  quadratic tree determined by $\alpha$.
  For each subset $\U$ of $Q(D)$, let $\O_\U =  \bigcap_{\alpha \in \U}\alpha$. In case $\U$ is the empty set, we define $\O_\U = F$.  
  
  \item If $\alpha \subseteq \beta$ in $Q(D)$, then $\beta$ is {\it proximate} to $\alpha$ if $\beta$ is a subring of the order valuation ring,   $\ord_{\alpha}$,  of $\alpha$.   Let $P(\alpha)$ denote the
  set of all $\beta$ proximate to $\alpha$.

  \item
   Let $\R(D)$ denote the set of rings of the form $\O_\U$ for some subset $\U$ of $Q(D)$. 
  
  \end{enumerate}
  \end{notation}

  Proposition~\ref{prop1.6} records properties of the quadratic tree $Q(D)$ and the set $X(D)$ of   
   minimal valuation overrings of $D$.
  
\begin{proposition} \label{prop1.6}   Assume Notation~\ref{note3.1}.  
\begin{enumerate}
\item  If $\alpha \in Q(D)$ properly contains $D$,  
then there exists a unique positive integer~$j$ such that $\alpha \in Q_j(D)$.
 The regular local 
  rings between $D$ and $\alpha$ are linearly ordered with respect to inclusion and form a chain of
  length $j$. 
\item
Each local ring $\alpha \in Q(D)$ is essentially finitely generated over $D$, that is,
there exist finitely many elements $x_1, \ldots, x_n \in \alpha$  such that $D[x_1, \ldots, x_n]_\p = \alpha$,
where $\p$ is the center of $\alpha$ on $D[x_1, \ldots, x_n]$.

 \item 
 If $\alpha,\beta \in Q(D)$, then the following are equivalent:
 \begin{enumerate}
 \item  $\alpha$ is a subring of $\beta$. 
 \item  $\alpha$ is in   the chain of regular local rings from $D$ to $\beta$.
 \item As points on the quadratic tree  $Q(D)$,  $\alpha \le \beta$.
 \end{enumerate}
\item
The valuation domains in 
 $X(D)$ are paired in a one-to-one correspondence  with the branches of the quadratic tree $Q(D)$.
 Each $V \in X(D)$ is the union of the quadratic sequence of $D$ along $V$.\footnote{In the terminology of 
 \cite{HLOST}, $V$ is a Shannon extension of $D$.}
\end{enumerate}

\end{proposition} 

\begin{proof}  These assertions follow from Remark~\ref{note2.1}.    
\end{proof}

\begin{remark} \label{rm3.3}
Let $R \in \R(D)$ and let 
$\U' = \{\alpha \in Q(D) ~|~ R \subseteq \alpha \}$.  
Let $\U$ 
denote the subset of $\U'$   of minimal points of  $\U'$.
Then the points in $\U$ are incomparable and $\O_\U = \O_{\U'}$.
Therefore: 
\begin{enumerate}
\item
Each ring $R \in \R(D)$ has the form 
$R = \O_\U$, where the $\alpha \in \U$ are incomparable and 
are minimal among points of $Q(D)$ that contain $R$. 
\item
If $D$ is Henselian and $\U$ is as defined in item 1,  
then it is shown  in \cite{HO} 
that the representation $R = \bigcap_{\alpha \in \U} \alpha$ 
is irredundant. 
\end{enumerate}
\end{remark}

\begin{remark}  \label{rm3.4}  
For $\alpha \in Q(D)$ the set $P(\alpha)$ of points proximate to $\alpha$ is by definition
$$
P(\alpha) ~ = ~  \{\beta \in Q(D) ~|~ \alpha \subseteq \beta \text{ and } \beta \subset \ord_{\alpha} =: V \}.
$$
The  points in $Q_1(\alpha)$ are all proximate to $\alpha$.  
For each $\beta \in Q_1(\alpha)$,  there exists 
a unique chain in $Q(\beta)$  of points in $P(\alpha)$  that may be described as follows: 
\begin{enumerate}
\item 
 The center 
of $V$ on $\beta$ is a height-one regular prime $\p = z\beta$, for
some $z \in \beta$. Let $w \in \m_{\beta}$ be 
such that $\m_{\beta} = (z, w)\beta$.  
Then $\beta[z/w]$ and $\beta[w/z]$ are affine components that define
 the 
blowup of $\m_{\beta}$.   Since $\beta[w/z]$ is not contained 
in $V$,  the only point of 
$Q_1(\beta)$ that is contained in $V$ is $\beta[z/w]_{(w, z/w)\beta}$,  and $V$ is 
centered on the height-one regular prime of this ring 
generated by $z/w$.
\item 
The process iterates to give as the union the rank 2 valuation domain contained in $V$ obtained 
by composite center  construction with  respect to the residue class  ring  $\beta/z\beta$.
\item
In summary,  the set $P(\alpha)$ consists of all the points in 
the first neighborhood $Q_1(\alpha)$ of $\alpha$ together with an
 infinite ascending ray emanating  from each 
point $\beta \in Q_1(\alpha)$.  The rays are in 
one-to-one correspondence with the points in $Q_1(\alpha)$.  

Corollary~\ref{noether down} implies that the set $P(\alpha)$  of points proximate to $\alpha$ is a 
Noetherian subspace of $Q(D)$ in the Zariski topology.

\item
It is a classical fact that an infinitely near point $\gamma$ to $D$
is proximate to at most 2 points of $Q(D)$. An argument for 
this  in a more general setting is  given in the proof of 
\cite[Lemma 2.7]{HLOST}.  

\item
{Let $\gamma \in Q(D)$ with $\gamma \ne D$.  Then there exists a unique point  $\alpha \in Q(D)$ such 
that $\gamma \in Q_1(\alpha)$. 
 The set $\E(\gamma)$ of essential valuation rings for
$\gamma$ contains $\ord_{\alpha}$ and 
at most one other  element that is  not in $\E(D)$ the set of essential valuation 
rings of $D$.  

\item
Let $V \in \E(D)$,  then $V = D_\p$,  where $\p$ is a height 1 prime of $D$.  The rings $\gamma$ 
 in $ Q(D) $   that are 
contained in $V$     are determined by the transform of $\p$ in $\gamma$ 
as defined by Lipman in \cite{L}.    The transform of $\p$ in $\gamma$ is either $\gamma$ or a 
height 1 prime $\q$,  where $\q \cap D = \p$ and $D_\p = \gamma_\q = V$.
If the transform of  $\p$ in $\gamma$ is  a prime ideal    $\q$,  then the finite sequence of 
 quadratic transforms from $D$  to $\gamma$ induces a finite sequence of local quadratic transforms of the 
   local domain $D/\p$ to the local domain $\gamma/\q$,  cf. \cite[Corollary II.7.15, p. 165]{H} and 
    \cite[Proposition 4.1]{GHOT}.

\item
The integral closure of $D/\p$ is a finite intersection of DVRs. 
 Associated to each of these DVRs, by composite construction, there exists a rank 2 valuation 
 overring $W$ of $D$.   Then $W \in X(D)$,  the set of minimal valuation overrrings of $D$.  Hence
 $W$ determines a unique branch in the quadratic tree $Q(D)$.   The points $\gamma \in Q(D)$ such
 that $\gamma \subset D_\p$ are the points in one of these branches. 
 
 \item 
 In summary,   if the integral closure of $D/\p$ has $t$ maximal ideals, then there exists for each 
 positive integer $n$ at most $t$ points in $Q_n(D)$ that are contained in $V = D_\p$. 
 There are $t$ branches in $Q(D)$ uniquely determined by $D_\p$,  
  and the points on these branches are precisely the points of $Q(D)$
 that are contained in $V = D_\p$. 
 }\footnote{If $D$ is not pseudo-geometric,  the integral closure of $D/\p$ may not be finite, but if the 
 integral closure of $D/\p$  has $t$ maximal ideals, then there is  a finite integral extension of $D/\p$ 
 that has $t$ maximal 
 ideals.}

\end{enumerate}

\end{remark}

\begin{discussion} \label{disc4.2}
For nonzero elements $f$ and $g$ in $D$,  the rational function $f/g$ has a  position in 
each  element of the quadratic tree $Q(D)$. We use the following terminology:
\begin{enumerate}
\item
$f/g$ has a  {\it zero}   at $\alpha \in Q(D)$ if $f/g \in \m_{\alpha}$,  the maximal ideal of $\alpha$.
\item
$f/g$ has a  {\it  pole}   at $\alpha \in Q(D)$ if $g/f   \in \m_{ \alpha} $.   
\item
$f/g$ is a  {\it unit } at $\alpha$ if both $f/g \in \alpha$ and $g/f  \in \alpha$.
\item 
$f/g$ is {\it undetermined}  at $\alpha$ if both $f/g \not\in \alpha$ and $g/f \not\in \alpha$.
\end{enumerate}
  It is clear
that if $f/g$  has a zero or pole, or is  a unit at  $\alpha$,  then $f/g$ has, respectively, a zero or pole or is 
 a unit in each $\beta \in Q(D)$ that dominates $\alpha$.

 Since the elements in $Q(D)$ are not valuation rings,  many rational functions  $f/g$ are such that 
 $f/g$ is undetermined at $\alpha$.     
 The rational function $f/g$ is undetermined at $\alpha$ if and 
 only if $g/f$ is undetermined at  $\alpha$. 
 In this case,    it is natural to  consider  $f/g$ in the points of $Q_1(\alpha)$. 
 
  Since $\alpha$ is a UFD, we may  assume that $f$ and $g$ have no common 
 prime  factors in $\alpha$.   A  prime factor   of $f$   in $\alpha$  
 generates  a  height 1 prime $\q$  of $\alpha$.  Since $D \subsetneq \alpha$,  there exists  a nonempty 
 finite set  of cardinality at most 2   of height 1 primes of $\alpha$ that contain $\m_D$.   If 
 $\q \cap D  = \m_D$, then $\alpha_\q$ is the order valuation ring of one of the points  $\beta$  in  the 
 finite sequence from $D$ to $\alpha$,  and $\alpha$ is proximate to $\beta$. 
 
 If $\q \cap D = \p$ is a height 1 prime of $D$,  then $D_\p = \alpha_\q$.  As noted in Remark~\ref{rm3.4}, 
 the finite sequence of 
 quadratic transforms from $D$  to $\alpha$ induces a finite sequence of local quadratic transforms of the 
   local domain $D/\p$ to the local domain $\alpha/\q$.  
This can be helpful in describing the zeros  and poles with respect to the points of $Q(D)$ 
of a rational function $f/g$.    

\end{discussion}


  Example~\ref{exam4.3} illustrates how to compute the position of a specific rational function
in the quadratic tree. Transforms of the ideal  $(f, g)D$ as defined  in \cite{L} are useful 
for such computations.      

\begin{example}  \label{exam4.3} 
Assume terminology as in Discussion~\ref{disc4.2}.  Let 
$$ 
f = xy \quad \text{  and  }  \quad g = y^2 + x^3.
$$
Then $f/g$  is undetermined  at $D$.  Every point of $Q_1(D)$ is a 
localization of either $D[y/x]$ or of $D[x/y]$,  and all the points
of $Q_1(D)$ other than  $D[y/x]_{(x, y/x)}$ and $D[x/y]_{(y, x/y)}$ are 
localizations of both  $D[y/x]_{(x, y/x)}$ and $D[x/y]_{(y, x/y)}$.

 Consider  
$D[\frac{y}{x}]$ and let $y_1 = \frac{y}{x}$.
Then $y = xy_1$ and 
$$
\frac{f}{g} ~ =   ~ \frac{xy}{y^2 + x^3} ~=   ~\frac{x^2y_1}{x^2y_1^2  + x^3}   ~ 
=    ~\frac{y_1}{y_1^2  + x}. 
$$  
It follows that $f/g$ is undetermined   at 
the point $\alpha := D[y/x]_{(x, y/x)}$,  the point in $Q_1(D)$ 
with maximal ideal generated by $(x, y_1)$. 

Consider $\beta : = \alpha[y/x^2]_{(x,  y/x^2)}$,  the point in $Q_2(D)$ with maximal 
ideal generated by $(x, y_2)$,   where  $y_2 = \frac{y_1}{x}$.   Then $y_1 = xy_2$ and
$$
\frac{f}{g} ~ =   ~ \frac{xy}{y^2 + x^3} ~ =    ~\frac{y_1}{y_1^2  + x}  ~=~ \frac{xy_2}{x^2y_2^2 + x}  ~
= ~ \frac{y_2}{xy_2^2 + 1}.
$$
It follows that $f/g$ has a zero at $\beta$.

Consider $D[x/y]$ and let $x_1 = x/y$.  Then $x = yx_1$  and
$$
\frac{f}{g} ~ = ~ \frac{xy}{y^2 + x^3} ~ = ~
\frac{y^2x_1}{y^2  + y^3x_1^3}  ~  =    ~\frac{x_1}{1   + yx_1^3}.
$$
If follows that $f/g$   has a zero at  $\gamma   : =  D[x/y]_{(y, x/y)} $,  the point in $Q_1(D)$ with maximal 
ideal generated by $(y, x_1)$.

We conclude that $\frac{f}{g}  = \frac{xy}{y^2 + x^3}$ has two  ``distinguished zeros" ,   namely 
$\beta \in Q_2(D)$ and 
$\gamma \in Q_1(D)$.   The points $\beta$ and $\gamma$ are incomparable and the points in the
quadratic tree $Q(D)$ at which $f/g$ has a zero are precisely the points that dominate either $\beta$
or $\gamma$.  

The local ring $R :=  D[f/g]_{(x, y, f/g)}$ is dominated by $\beta$ and $\gamma$.   The 
integral closure of the ideal $I = (xy,   y^2 + x^3)D$ is the product $J$  of the simple complete
ideals $(x, y)D$ and $(y, x^2)D$.   The complete ideal  $J = (xy, y^2, x^3)D$  has a saturated factorization
as defined in \cite[Definition 5.11]{HJLS},  
and the projective model $\Proj D[Jt]$ is nonsingular with $\beta$ and $\gamma$ points on this model.

The integral closure of $R$  is $R[\frac{y^2}{y^2 + x^3},   \frac{x^3}{y^2 + x^3}] =   \beta \cap \gamma$. 

To compute the poles of $f/g$, define $S = D[\frac{y^2 + x^3}{xy}]_{(x, y, \frac{y^2 + x^3}{xy})}$.   The rings
in $Q(D)$ that dominate the local ring $S$ are the zeros of $g/f$ and therefore the poles of $f/g$. 
The affine component $ A : = D[\frac{y^2}{xy}, \frac{x^3}{xy}]$ of $\Proj D[Jt]$ contains
 $D[\frac{y^2 + x^3}{xy}]  = D[\frac{y}{x} + \frac{x^2}{y}]$ as 
a subring.  The points in the quadratic tree $Q(D)$ at which $f/g$ has a pole are precisely 
the points that dominate $\delta : = D[\frac{y}{x}][\frac{x^2}{y}]_{(\frac{y}{x},  \frac{x^2}{y})}$.
Notice that $\delta \in Q_2(D)$ and $\delta$ is the integral closure of $S$. Thus $\delta$ is the 
unique ``distinguished pole''  of $f/g$.

\end{example} 

\section{The patch topology of $Q^*(D)$}  \label{sec4}

In this section and the next we examine topological properties of the tree $Q(D)$ and the partially ordered set $Q^*(D)$ consisting of $Q(D)$ and the valuation rings birationally dominating $D$.  We first describe the patch topology of $Q^*(D)$, a topology that is finer than the Zariski topology. As follows in Remark~\ref{recover}, properties of the Zariski topology, which are the focus of the next section, can be derived from the patch topology, so our approach in this section is to focus on the patch limit points of subsets of  $Q^*(D)$ and use this description  in the next section to describe properties of the Zariski topology of $Q^*(D)$.  
 The patch topology is a common tool for studying the Zariski-Riemann space of valuation rings of a field; see for example \cite{FFL, FFL2, Kuh, OZar, OTop, OPrin, OTop2}.
 
 Our methods in this section require us to work occasionally not just in $Q^*(D)$ but also in the set of all local rings birationally dominating $D$.  We formalize our notation for this section and the next as follows.

 \begin{notation} \label{not1.3}
 Let $D$ be a regular local ring of dimension 2  and 
 let $F$ denote the quotient field 
 of $D$.
 \begin{enumerate}

 \item
 Let $L(D)$ denote the set of all local rings that birationally dominate $D$. Similarly, for $\alpha \in Q(D)$, $L(\alpha)$ is the set of local rings that birationally dominate $\alpha$.  

\item For each $x_1,\ldots,x_n \in F$, we let  
 $$
{\mathcal{U}}(x_1,\ldots,x_n)=\{R \in L(D) ~|~ x_1,\ldots,x_n \in R\},
$$  
and
$$
{\mathcal{V}}(x_1,\ldots,x_n):=\{R \in L(D)~|~x_1,\ldots,x_n {\mbox{ are not all in }}  R\}.$$ 
Thus $\V(x_1, \ldots, x_n)  = L(D) ~\setminus ~ \U(x_1, \ldots , x_n)$.

\item The {\it Zariski topology} on $L(D)$ is the topology having a basis of open sets given by the sets  ${\mathcal{U}}(x_1,\ldots,x_n)$, where $x_1,\ldots,x_n \in F$.  
 Thus a basis of closed sets is given by sets  
$
{\mathcal{V}}(x_1,\ldots,x_n)$, where $x_1,\ldots,x_n \in F$.

 \item {As in Notation~\ref{note2.1}.6}, we denote by 
 $X(D)$ the subset of $L(D)$ consisting  of the minimal valuation
 overrings of $D$. 
 
 \item We let $Q^*(D)$ be the union of $Q(D)$ and the set of valuation rings in $L(D)$. Thus $Q^*(D)$ is the union of the set $Q(D)$, the set $X(D)$ and the set of prime divisors of $D$ of the second kind. 
 
 \item {We extend the definition of $\O$ given in Notation~\ref{note3.1}.3 to $Q^*(D)$ by defining, for each
  subset $\U$ of $Q^*(D)$,  $\O_\U$ to be the intersection of the rings in $\U$. If $\U$ is empty, then 
  we set $\O_\U = F$.  }


 \end{enumerate}

 \end{notation} 

The set $Q^*(D)$ is partially ordered by inclusion. While the subset $Q(D) \cup X(D)$ of $Q^*(D)$ is a tree by Proposition~\ref{prop1.6}.4, the set $Q^*(D)$ is not. This is because a prime divisor of $D$ 
of the second kind contains infinitely  many  valuation rings in $X(D)$.

We prove in Corollary~\ref{Q* spectral} that $Q^*(D)$ with the Zariski topology is a  {\it spectral space}.  This means  that it is $T_0$ and quasicompact, it has a basis of quasicompact open sets closed under finite intersections, and every irreducible closed set has a (unique) generic point\footnote{A  theorem of Hochster \cite[Corollary, p.~45]{Hoc} shows that a topological space is spectral if and only if it is homeomorphic to the prime spectrum of a ring.}. 
This is useful not only because it situates the topology of $Q^*(D)$ in an appropriate context,
but because  it   also allows us to use some of the tools for working with spectral spaces. 

The proof that $Q^*(D)$ is a spectral space involves proving that it is a patch closed subset of a larger spectral space, namely $L(D)$.  We recall that the   {\it patch topology}\footnote{The patch topology is sometimes referred to as the constructible or Hausdorff topology of a spectral space.}  of a spectral space $X$ has as a basis of open sets the sets that are an intersection of a quasicompact open set and the complement of a quasicompact open set. 
The patch topology of a spectral space is quasicompact, Hausdorff and has a basis of sets that are both closed and open; {see \cite[Theorem 1]{Hoc}.}

\begin{proposition} \label{L(D) spectral} 
The space $L(D)$  with the Zariski topology is a spectral space having as a basis of quasicompact open sets the sets of the form $\U(x_1,\ldots,x_n)$, where $x_1,\ldots,x_n \in F$.  
\end{proposition} 

\begin{proof} For the purpose of this proof, we introduce additional notation that will not be needed later. We denote by  $L'(D)$ the set of  local overrings of $D$.  Thus $L'(D)$ consists of the rings in $L(D)$ and the local overrings of $D$ that do not dominate $D$; these latter rings are precisely the essential valuation rings of $D$.
 For $x_1,\ldots,x_n \in F$, we let 
$$\
\U'(x_1,\ldots,x_n)= \{R \in L'(D) ~|~ x_1,\dots,x_n \in R\}.
$$ 
 Thus $\U(x_1,\ldots,x_n) = \U'(x_1,\ldots,x_n) \cap L(D)$. 
 The sets of the form $ \U'(x_1,\ldots,x_n)$ yield a basis for a topology on $L'(D)$, which we refer to as the Zariski topology on $L'(D)$.  Viewed as a subspace of $L'(D)$, $\U'(x_1,\ldots,x_n)$ is a spectral space (see  \cite[Corollary 2.14]{FFS} or \cite[Example 2.2(7)]{OTop}). In particular $\U'(x_1,\ldots,x_n)$ is quasicompact. Since $L'(D) = \U'(1)$, we have also that $L'(D)$ is a spectral space. 
 We make use of these facts in what follows.

We claim that for  $x_1,\ldots, x_n \in F$, the subspace $\U(x_1,\ldots,x_n)$ of $L(D)$ is a patch closed subset of the spectral space $L'(D)$.   
To prove this, it suffices to show that each patch limit point of $\U(x_1,\ldots,x_n)$ in $L'(D)$ is 
in $\U(x_1,\ldots,x_n)$.  
 Let $R$ be a patch limit point of  $\U(x_1,\ldots,x_n)$ in $L'(D)$. To prove that $R$
 is in $\U(x_1,\ldots,x_n)$, it suffices to show $R$
  dominates $D$ and $x_1,\ldots,x_n \in R$.   
  
  
  To prove that $R$ dominates $D$, let
 $R^\times$ denote the set of units in $R$.  Then the set $${\mathcal{U}}':=\bigcup_{u \in R^\times} \left( {\mathcal{U}}'(u) \cap {\mathcal{U}}'(u^{-1}) \right) $$ is open in the patch topology of $L'(D)$ and contains $R$.  Since $R$ is  a patch limit 
  point of $\U(x_1,\ldots,x_n)$ in $L'(D)$,  there exists $\alpha \in \U(x_1,\ldots,x_n)$ such that  
  $\alpha \in {\mathcal{U}}'$.  Thus every unit in $R$ is a unit in $\alpha$.  
  If $R$ does not dominate $D$, then there exists a nonunit $d \in D$ such that $d$ is a unit in $R$, hence also in $\alpha$.  But $\alpha$ dominates $D$, so this is impossible. We conclude that $R$ dominates $D$.
  
  To see that $x_1,\ldots,x_n \in R$, suppose to the contrary that $x_i \not \in R$ for some $i$. 
  Then $R \in {\mathcal V}' := \{S \in L'(D) ~|~    x_i \not \in S\}$.  Since $R$ is a patch limit point 
  of $\U(x_1,\ldots,x_n)$ in $L'(D)$ and ${\mathcal V}'$ is a patch open set in $L'(D)$, the  
  intersection  $\U(x_1,\ldots,x_n) \cap {\mathcal V}'$ is nonempty, a contradiction. Therefore  $x_1,\ldots,x_n \in R$, which proves that $\U(x_1,\ldots,x_n)$ is patch closed in $L'(D)$.
 A patch closed subspace of a spectral space is a spectral space \cite[Tag 0902]{stacks-project}, so  
 $\U(x_1,\ldots,x_n)$ is a spectral  space and hence is quasicompact. Since $L(D) = \U(1)$,
  Proposition~\ref{L(D) spectral}
 now follows.  
\end{proof} 


\begin{corollary} \label{basis} The patch topology of the spectral space $L(D)$ has as a basis of 
 closed and open sets the sets of  the form $$\left({\mathcal{U}}(S_1) \cup \cdots \cup {\mathcal{U}}(S_n)\right) \cap {\mathcal{V}}(T_1) \cap \cdots \cap {\mathcal{V}}(T_m), $$ where $S_1,\ldots,S_n,T_1,\ldots,T_m$ are finite subsets of $F$.  
 \end{corollary} 
 
 \begin{proof} Since the {patch topology}  on a spectral space has a basis of open sets 
 the sets that are an intersection of a quasicompact open set and the complement of a quasicompact open set, this follows from Proposition~\ref{L(D) spectral}. 
 \end{proof}

 \begin{remark}  \label{recover} Although the patch topology on $L(D)$ is finer than the Zariski topology, the Zariski topology can be recovered from it: The Zariski closed sets $\V$ of $L(D)$ are precisely the downsets (with respect to set inclusion) of the patch closed subsets of $L(D)$. This follows for example from \cite[Corollary, p.~45]{Hoc}.
 The upsets of patch closed subsets of $L(D)$ form a basis for the {\it inverse topology} on $L(D)$. 
 \end{remark}

Viewing $Q(D)$ as a subspace 
 of $L(D)$, we describe in Theorem~\ref{patch closure} the
  {\it patch closure} of $Q(D)$ in $L(D)$, i.e., the closure 
  of $Q(D)$ in $L(D)$ with respect to the patch topology. 
  To do this, we first describe in the next lemmas additional  
  topological properties of $Q(D)$ and the  limit points of $Q(D)$ in the patch topology.

\begin{lemma} \label{clopen}  For each $\alpha \in Q(D)$, the subspace 
 $L(\alpha)$ is patch clopen in $L(D)$. 
\end{lemma}

\begin{proof}  
By Proposition~\ref{prop1.6}.2, there exist 
$x_1,\ldots,x_n \in \alpha$ such that $\alpha = D[x_1,\ldots,x_n]_\p$ 
for some prime ideal $\p$ of $D[x_1,\ldots,x_n]$.  Write $\p = (y_1,\ldots,y_m)D[x_1,\ldots,x_n]$,
where $y_1, \ldots, y_m$ are nonzero.  Since $\alpha$ has dimension $2$, 
$\p$ is a maximal ideal of $D[x_1,\ldots,x_n]$. 
 Since $L(\alpha)$ is the set of rings in $L(D)$ that dominate $\alpha$, 
 it follows that 
 $$L(\alpha) =  {\mathcal{U}}(x_1,\ldots,x_n) \cap {\mathcal V}(1/y_1) \cap \cdots \cap {\mathcal{V}}(1/y_m).$$  Thus $L(\alpha)$ is patch clopen in $L(D)$.
\end{proof} 

In light of Remark~\ref{recover}, the following notation will be useful in this section and the next. 

 \begin{notation} \label{downset}
For a  nonempty subset $\S$ of $Q^*(D)$, we denote by  $\downarrow \hspace{-.03in} \S$ the {\it downset} 
of $\S$ in $Q(D)$: 
\begin{center}
 $\downarrow \hspace{-.03in} \S = \{\alpha \in Q(D) \mid \alpha \subseteq \beta$ for some $\beta \in \S\}$.  
\end{center} 
When $\S $ consists of a single ring $R$, we write $\downarrow \hspace{-.03in} R$ for $\downarrow \hspace{-.03in} \{R\}$.  
Note that while $\S$ may be a subset of $Q^*(D)$,  in our notation the members of the downset are restricted to $Q(D)$.  To differentiate between these two cases, we set  \begin{center}
 $\downarrow^* \hspace{-.02in} \S = \{\alpha \in Q^*(D) \mid \alpha \subseteq \beta$ for some $\beta \in \S\}$.  
\end{center} 

 \end{notation}

Theorem~\ref{minimal V}  
describes  the valuation rings that are  patch limit points in $L(D)$ of a subset $\S$ of $Q(D)$.  
Theorem~\ref{patch closure} implies patch limits points of subsets of $Q(D)$ are necessarily valuation rings in $Q^*(D)$.   This description shows that the patch limit points of $\S$ are determined by the properties of the partially ordered set $Q^*(D)$ rather than the nature of the underlying objects that comprise the set $Q^*(D)$.  

\begin{theorem} \label{minimal V} Let $\S$ be a subset of $Q(D)$, and let $V$ be a valuation ring in $Q^*(D)$. 
\begin{enumerate} 
\item Suppose $V \in X(D)$.  Then the following are equivalent. 
\begin{enumerate}
\item $V$ is  a patch limit point of $\S$ in $L(D)$.
\item  $\S \cap Q(\alpha)$ is nonempty for each $\alpha \in Q(D)$ dominated by $V$.  
 \end{enumerate}
 \item  Suppose  $V \not \in X(D)$. Then  $V$ is a prime divisor of the second kind and an order valuation ring of  some $\alpha \in Q(D)$.  In this case, the following are equivalent.
 \begin{enumerate}
 \item 
 $V$ is a patch limit point of $\S$ in $L(D)$.
 \item $Q_1(\alpha) \: \cap \downarrow \hspace{-.03in} \S$ is infinite.
 
 \item There are infinitely many incomparable rings in $\S$ contained in $V$.
 
 \item There are infinitely many incomparable rings in  $\S$ proximate to $\alpha$. 
 \end{enumerate}

 \end{enumerate}
\end{theorem}

\begin{proof}  
To prove item 1, suppose that $V \in X(D)$ is a patch limit point of $\S$ in $L(D)$ and that $\alpha \in Q(D)$ is dominated by $V$.  Since $V \in X(D)$, Remark~\ref{note2.1}.8.b implies that there exists  $\alpha_1 \in Q_1(\alpha)$ such that $V$ dominates $\alpha_1$.  Let $\U = L(\alpha) \setminus L(\alpha_1)$. By Lemma~\ref{clopen}, $\U$ is a patch open set of $L(D)$ containing $V$. As $V$ is  a patch limit point of $\S$ in $L(D)$, the set $\U \cap \S$ is nonempty. Since $\S \subseteq Q(D)$ and $Q(\alpha) = Q(D) \cap L(\alpha)$, we conclude that $Q(\alpha) \cap \S$ is nonempty.

 Conversely, suppose $Q(\alpha) \cap \S$ is nonempty for each $\alpha$ dominated by $V$.  Since $V \in X(D)$, Remark~\ref{note2.1}.8.b implies there is an infinite sequence $$\alpha_0 \subseteq \alpha_1 \subseteq \cdots \subseteq  \alpha_{i} \subseteq \cdots $$ of elements of $Q(D)$ dominated by $V$ such that $V = \bigcup_{i} \alpha_i$.  
 By assumption,  $\S \cap Q(\alpha_i)$ is nonempty for each $i \geq 1$.  Let ${\mathcal{U}}$ be a patch open subset of $L(D)$ containing $V$.  {By Corollary~\ref{basis}} there are nonzero $x_1,\ldots,x_n,y_1,\ldots,y_m \in F$ such that $$V \in {\mathcal{U}}( x_1,\ldots,x_n) \cap {\mathcal{V}}(y_1) \cap \cdots \cap  {\mathcal{V}}(y_m) \subseteq \U.$$
Since $V$ is a valuation ring, we have $y_1^{-1},\ldots,y_m^{-1} \in {\ff M}_V$.
 Since $V$ is a directed union of the $\alpha_i$, there is $i \geq 0$ such that $x_1,\ldots,x_n,y_1^{-1},\ldots,y_m^{-1} \in \alpha_i$.  Moreover, since $V$ dominates $\alpha_i$, the maximal ideal of $\alpha_i$ contains $y_1^{-1},\ldots,y_m^{-1}$.  {Since the rings in $Q(\alpha_i)$ dominate $\alpha_i$, it follows that $$Q(\alpha_i) \subseteq {\mathcal{U}}( x_1,\ldots,x_n) \cap {\mathcal{V}}(y_1) \cap \cdots \cap  {\mathcal{V}}(y_m) \subseteq  {\mathcal{U}}.$$}Thus  the assumption that $\S \cap Q(\alpha_i)$ is nonempty implies  that  $\S \cap {\mathcal{U}}$ is nonempty, which proves that $V$ is  a patch limit point of $\S$ in $L(D)$.  

It remains to prove item 2. Since $V \not \in X(D)$, $V$ is a prime divisor of $D$ of the second kind. Therefore, by Remark~\ref{note2.1}.7, there is $\alpha \in Q(D)$ such that $V$ is the order valuation ring of $\alpha$.

(a) $\Rightarrow$ (b) 
 Suppose $V$ is a patch limit point of $\S$.  We   show that there are infinitely many   local  rings in $Q_1(\alpha)$ that are dominated by local rings in $\S$. To this end, let  $\beta_1,\ldots,\beta_n \in Q_1(\alpha)$. 
By Lemma~\ref{clopen} and the fact that points are closed in the patch topology, we have  that
$${\mathcal{U}}:=L(\alpha) \setminus (L(\beta_1) \cup \cdots \cup L(\beta_n) \cup \{\alpha\})$$ is a {nonempty} patch open subset of $L(D)$.  Since $V$ is the order valuation ring of $\alpha$, $V$ dominates $\alpha$ but does not dominate any local ring in $Q_1(\alpha)$.  Consequently, $V \in \U$.  
 Since $V$ is a patch limit point of $\S$, there exists $\gamma \in \U \cap \S$.  Because $\gamma $ properly dominates $\alpha$ and does not dominate any of the $\beta_i$, it follows that $\gamma$ dominates some member of $Q_1(\alpha)$ distinct from $\beta_1,\ldots,\beta_n$.  
 Since this is true for any finite subset $\{\beta_1,\ldots,\beta_n \}$ of $ Q_1(\alpha)$,   we conclude  there are infinitely many local rings in $Q_1(\alpha)$ dominated by local rings in $\S$.

 (b) $\Rightarrow$ (c) This is clear since the rings in $Q_1(\alpha)$ are incomparable and contained in $V$.
 
 (c) $\Rightarrow$ (d) The rings $\beta$  in $Q(D)$ contained in $V$ are either proximate to $\alpha$ 
or  are dominated by $V$ (see Remark~\ref{note2.1}.7). In the latter case, $\alpha$ dominates $\beta$, and hence there are only finitely many rings in $Q(D)$ dominated by $V$.  Item d now follows from item c. 
 
(d) $\Rightarrow$ (b)  This follows from Remark~\ref{rm3.4}. 


(b) $\Rightarrow$ (a) 
Suppose that the set of local rings in $Q_1(\alpha)$  that are dominated by local rings in $\S$ is infinite.  We show $V$ is  a patch limit point of $\S$. 
Let $X$ be the projective model of $\Spec \alpha$ obtained by blowing up the maximal ideal of $\alpha$. Then  
 $Q_1(\alpha)$ is the set of closed points of $X$ and $V$ is the generic point of the irreducible Zariski closed subset ${\mathcal{C}}:= \{V\} \cup Q_1(\alpha)$ of $X$.   Since ${\mathcal{C}}$ is a Noetherian space in the Zariski topology, every closed subset of ${\mathcal{C}}$ is a finite union of irreducible closed sets.  The only irreducible proper closed subsets of ${\mathcal{C}}$ are the singleton subsets of  $Q_1(\alpha)$, so
 it follows 
 that any infinite subset of ${\mathcal{C}}$ is dense in ${\mathcal{C}}$.    
 By assumption, there is an infinite subset $\T$  of $Q_1(\alpha)$ such that each local ring in $\T$  is dominated by a local ring in $\S$.  Since $\T$ is Zariski dense in ${\mathcal{C}}$ and $V$ is the generic point for ${\mathcal{C}}$, it follows that 
 $V$ is a patch limit point of $\T$ \cite[Proposition~2.6(1)]{OZar}. 
 
 To see now that $V$ is  a patch limit point of $\S$, let \begin{center}
 $\S' = \{\sigma \in \S  ~|~ \tau \subseteq \sigma$ for some $\tau \in \T\}$. 
  \end{center}  
  Fix $y \ne 0 $ in the maximal ideal of $ V$.  
 For all nonzero $x_1,\ldots,x_n \in V$,
 we have $$V \in {\mathcal{U}}(x_1,\ldots,x_n,y) \cap {\mathcal{V}}(1/y).$$ 
 Since ${\mathcal{U}}(x_1,\ldots,x_n,y) \cap {\mathcal{V}}(1/y)$ is patch open and $V$ is a patch limit point of $\T$,  there exists $\tau \in  \T \cap {\mathcal{U}}(x_1,\ldots,x_n,y) \cap {\mathcal{V}}(1/y).$ Let $\sigma \in \S'$ such that $\tau \subseteq \sigma$. Since $\sigma$ dominates $\tau$, we have  $\sigma  \in {\mathcal{U}}(x_1,\ldots,x_n,y) \cap {\mathcal{V}}(1/y)$. This shows that for all nonzero $x_1,\ldots,x_n \in V$, we have $$\S' \cap {\mathcal{U}}(x_1,\ldots,x_n,y) \cap {\mathcal{V}}(1/y) \ne \emptyset.$$  Since the patch closure $\overline{\S'}$ of $\S'$ in $L(D)$ is compact in the patch topology and the collection of patch closed sets of the form  $\overline{\S'} \cap {\mathcal{U}}(x_1,\ldots,x_n,y) \cap {\mathcal{V}}(1/y)$ has the finite intersection property,  
 there is a ring $W \in \overline{\S'}$ such that $W \in  {\mathcal{U}}(x_1,\ldots,x_n,y) \cap {\mathcal{V}}(1/y) $ for all  $x_1,\ldots,x_n \in V$. Thus $V \subseteq W \subsetneq F$.  Since $V$ is a DVR, we conclude that $V = W$, which proves that $V \in \overline{\S'}$. Since $V$ is not in $\S'$, it must be that $V$ is  a patch limit point of $\S'$, hence  a patch limit point of $\S$.   
\end{proof}


Using the description of patch limit points in Theorem~\ref{minimal V}, we  next 
show  
that the patch limit points of the 
set $Q(D)$ in $L(D)$  is the set of all valuation rings that birationally dominate $D$.  

\begin{theorem} \label{patch closure} 
The set of patch limits points  of $Q(D)$ in $L(D)$ is precisely the set of  valuation rings in $L(D)$.  
Thus the
 patch closure of $Q(D)$ in $L(D)$ is  ${Q}^*(D)$. 
\end{theorem}

\begin{proof}  By Proposition~\ref{L(D) spectral}, $L(D)$ is a spectral space.  
First, we prove that every patch limit point of $Q(D)$ in $L(D)$ is a valuation ring. Let $R \in L(D)$ be a patch limit point of $Q(D)$, and let $V$ be a valuation ring in $L(D)$ dominating $R$. By Remarks~\ref{note2.1}.8, the rings in $Q(D)$ dominated by $R$ form a chain under inclusion. If this chain is infinite, then  by Remarks~\ref{note2.1}.8, we conclude that $R = V$, and the proof is complete.

Suppose there are only finitely many rings in $Q(D)$ dominated by $R$, and choose $\alpha \in Q(D)$ such that no member of $Q(D)$ properly dominates $\alpha$ and is dominated by $R$.  
Write ${\ff m}_\alpha = (x,y)\alpha$.  We claim that $x/y \in R$ or $y/x \in R$.  Suppose to the contrary that $x/y \not \in R$ and $y/x \not \in R$. Then 
 $R \in L(\alpha)  \cap \V(x/y) \cap \V(y/x)$. By Lemmas~\ref{basis} and~\ref{clopen}, the set $L(\alpha)  \cap \V(x/y) \cap \V(y/x)$is patch open in $L(D)$. Since $R$ is a patch limit point   of $Q(D)$ and the patch topology is Hausdorff, this implies that the set $L(\alpha)  \cap \V(x/y) \cap \V(y/x) \cap Q(D)$ is infinite. Since $Q(\alpha) = L(\alpha) \cap Q(D)$, we have then that $Q(\alpha)  \cap \V(x/y) \cap \V(y/x)$ is infinite. But since $x,y$ is  a system of regular parameters for $\alpha$, it follows that $Q(\alpha)  \cap \V(x/y) \cap \V(y/x) = \{\alpha\}$, a contradiction. 
Therefore  $x/y \in R$ or $y/x \in R$. 

Assume $x/y \in R$.  The center of $R$ in the ring $\alpha[x/y]$ contains the prime ideal $y\alpha[x/y]$. If $R$ is centered on a prime ideal properly containing $y\alpha[x/y]$, then $R$ is centered on a ring in $Q_1(\alpha)$, contrary to the choice of $\alpha$. Thus $R$ 
 is centered on  $y\alpha[x/y]$. Since the localization of $\alpha[x/y]$ at this height one prime ideal is the order valuation ring of $\alpha$, we conclude that $R$
 contains the order valuation of $\alpha$. Therefore the order valuation ring is a DVR birationally dominated by $R$, so  it follows that $R$ is the order 
 valuation ring of $\alpha$, proving that $R$ is a valuation ring.  This shows that the patch limit points of $Q(D)$ in $L(D)$ are valuation rings.

Finally, 
 every valuation ring $V$ in $L(D)$ is a patch limit point of $Q(D)$, 
  since if  $V$ is a minimal valuation ring of $D$, then $V$ is a patch limit point of $Q(D)$ by Theorem~\ref{minimal V}.1, while if $V$ is a prime divisor of $D$ of the second kind, then $V$ is a patch limit point of $Q(D)$ by Theorem~\ref{minimal V}.2.
 \end{proof}
 
 \begin{corollary} \label{Q* spectral} $Q^*(D)$ is a spectral space with respect to the Zariski topology. The patch topology of this spectral space is the subspace topology induced by the patch topology of $L(D)$.  
 \end{corollary}
 
 \begin{proof}
 As a patch closed subset of the spectral space $L(D)$, $Q^*(D)$ is a spectral space with respect to the Zariski topology \cite[Tag 0902]{stacks-project}. 
 That the  
  patch topology of the spectral space $Q^*(D)$ is the subspace topology of $Q^*(D)$ with respect to the  patch topology of $Q^*(D)$ follows from the discussion after Theorem~1 in
 \cite[p.~45]{Hoc}.
 \end{proof}
 
 \begin{corollary} \label{cor4.10}
  Every infinite subset of $Q^*(D)$ has a patch limit point in $Q^*(D)$. 
 \end{corollary}
 
 \begin{proof} Let $\S$ be an infinite subset of $Q^*(D)$. By Corollary~\ref{Q* spectral}, $Q^*(D)$ is a spectral space, so $Q^*(D)$ is quasicompact in the patch topology \cite[Theorem 1]{Hoc}. Therefore the patch closure of $\S$ in $Q^*(D)$ is infinite and quasicompact in the patch topology, hence not discrete. Consequently, the patch closure of $\S$ contains a patch limit point. 
 \end{proof}

\begin{corollary} \label{discrete} 
The set of patch isolated points of $Q^*(D)$ is  $Q(D)$. Thus $Q(D)$ is discrete  in the subspace topology induced by the patch topology of $Q^*(D)$.  \end{corollary} 

\begin{proof} 
Let $\alpha \in Q(D)$.   By Lemma~\ref{clopen}, $L(\alpha)$ is patch clopen in $L(D)$.  
 Since 
 $Q^*(\alpha) = Q^*(D) \cap L(\alpha),$ we have that $Q^*(\alpha)$ is clopen in the subspace topology of  $Q^*(D)$ induced by the patch topology on $L(D)$.  
 To see now that $\alpha$ is a patch isolated point in $Q^*(D)$, let $s,t$ be a system 
  of regular parameters for $\alpha$.  Since every local ring in $Q^*(\alpha) \setminus \{\alpha\}$ contains either $\alpha[s/t]$ or $\alpha[t/s]$, we have $$\{\alpha\} = Q^*(\alpha) \cap {\mathcal{V}}(s/t) \cap {\mathcal{V}}(t/s).$$ Therefore  $\{\alpha\}$ is  open in the subspace topology of $Q^*(D)$ induced by the patch topology of $L(D)$, and hence by Corollary~\ref{Q* spectral}, $\alpha$ is a patch isolated point of $Q^*(D)$.   Since every valuation ring in $Q^*(D)$ is by Theorem~\ref{patch closure} a patch limit point of $Q(D)$, Corollary~\ref{discrete} now follows. 
\end{proof}

\begin{discussion} We summarize some of the key features of the patch topology on $Q(D)$ and $Q^*(D)$.  The patch topology on $Q^*(D)$ is  Hausdorff and has a basis of clopens (i.e., the patch topology is zero-dimensional). With the subspace topology induced by the patch topology, $Q(D)$ is a discrete space (Corollary~\ref{discrete}), and hence the patch topology is not interesting intrinsically for $Q(D)$. With respect to the patch topology, it is only when viewed as a subspace of $Q^*(D)$ that $Q(D)$ becomes topologically interesting. 
Indeed, 
 the set of patch limit points of any infinite subset $\S$ of $Q(D)$ must consist exclusively of valuation rings in $Q^*(D)$ (Corollaries~\ref{cor4.10} and~\ref{discrete}). Which valuation rings are patch limit points of $\S$ is determined by the configuration of $S$ within the tree $Q(D)$; see Theorem~\ref{minimal V}.  The calculation of these limit points is an important tool in the next section when we work with the Zariski topology on $Q(D)$ and $Q^*(D)$.  
\end{discussion}

As an application of some of the ideas in this section, we describe 
the  structure of an intersection of order valuation rings of rings below a fixed level in $Q(D)$.  Recall that an integral domain $R$ is an {\it almost Dedekind domain} if for each maximal ideal $M$ of $R$, $R_M$ is a DVR.  

\begin{corollary} \label{aDd} Let $n >0$, and let $\S$ be a nonempty subset of $Q(D)$ consisting of rings of level at most $n$. Then the intersection of the order valuation rings of the rings in $\S$ is an almost Dedekind domain.
\end{corollary}

\begin{proof}  
Since every ring properly between an almost Dedekind domain and its quotient field is almost Dedekind, 
it suffices to prove  the lemma in the case in which $\S$ is the set of all of rings in $Q(D)$ of level  at most $n$. 
By Theorems~\ref{minimal V} and~\ref{patch closure}, the patch closure of $\S$ consists of $\S$ and  the set $\T$ of all prime divisors $V$ of $D$ of the second kind such that $V$ is the order valuation ring of a ring in $Q(D)$ of level {at most} $n-1$.  As the set of patch limit points of $\S$,  $\T$ is a patch closed subset of $Q^*(D)$. 
{By Corollary~\ref{Q* spectral}, $Q^*(D)$ is a spectral space and hence quasicompact in the patch topology \cite[Theorem~1]{Hoc}. 
Therefore, as a patch closed subspace of $Q^*(D)$, $\T$ is quasicompact in the patch topology. Since the patch topology is finer than the Zariski topology, $\T$ is quasicompact in the Zariski topology. 
The intersection of rings in a Zariski quasicompact set of DVRs of a field having the property that the intersection of the maximal ideals of these DVRs is nonzero is an almost Dedekind domain \cite[Corollary 5.8]{OPrin}. Therefore  
the intersection of the order valuation rings of the rings in $\S$
is an almost Dedekind domain.}
\end{proof} 

Corollary~\ref{aDd} need not remain valid if the restriction  on the levels of the rings in $Q(D)$ is removed. For example, the intersection of the order valuations rings of all the rings in $Q(D)$ (i.e., the intersection of all prime divisors of the second kind) is $D$. {To see this, suppose that  $x \in F \setminus D$. If $x^{-1} \in D$, then no prime divisor of the second kind contains $x$. Otherwise, if $x^{-1} \not \in D$, then no prime divisor of the second kind  of $D[x^{-1}]_{({\ff m}_D, x^{-1})D}$ contains $x$.  In either case, we find a prime divisor of the second kind not containing $x$, from which it follows that  
$D$ is the intersection of the prime divisors of the second kind.  In contrast with the situation in Corollary~\ref{aDd}, the set of prime divisors of the second kind of $D$ (with no restriction on level) is not quasicompact, and hence we may not appeal to \cite[Corollary 5.8]{OPrin} as we did in the proof of Corollary~\ref{aDd}.  
}

From Corollary~\ref{aDd} we deduce a representation theorem for rings in $\R(D)$ obtained as an intersection of rings of level at most $n$ in $Q(D)$.  

\begin{corollary} \label{cor4.13}
Let $n >0$, and let $\U$ be a nonempty subset of $Q(D)$ consisting of rings of level at most $n$. Then $\O_\U$ is  the  intersection of a  PID that is a 
localization of $D$ and an almost Dedekind overring. 
\end{corollary} 

\begin{proof} 
If $\O_\U = D$, then for a prime element $z$ of $D$, $\O_\U = D[1/z] \cap D_{zD}$ is  the
 intersection of a  PID that is a localization of $D$  and a DVR. 
 Suppose that $\O_\U \ne D$.   
 Let $A$ be the intersection of all  the prime divisors of  the  first kind that contain $\O_\U$. 
 Since $A \ne D$, $A$ is a PID that is a localization of $D$ obtained by inverting the prime 
 elements associated to prime divisors of the first kind that do not contain $A$.
Each $\alpha \in \U$ is an intersection of prime divisors of $D$ of the first kind and 
at most 2  prime divisors of $D$ of the second kind, each of which is an order valuation ring of 
some $\beta \in Q(D)$ of level less 
 than the level of $\alpha$,  see {the discussion after} Remark~\ref{rm3.4}.
 Let $B$ be the intersection of these valuation rings. 
 By Corollary~\ref{aDd}, $B$ is an almost Dedekind domain. Since $\O_\U = A \cap B$, 
 this proves  Corollary~\ref{cor4.13}.
\end{proof}

\section{The Zariski topology of $Q(D)$}   \label{sec5}

We describe in this section the Zariski topology of    $Q(D)$ and $Q^*(D)$ using the fact from Remark~\ref{recover} that the Zariski closure of a subset  $\S$ of $Q^*(D)$ is the downset of the patch closure of  $\S$.  For the purposes of several results in this section, we need the following notation to distinguish among the patch limit points of a subset $\S$ of $Q(D)$ those that are prime divisors of the second kind.  

 \begin{notation} \label{S infinity}
Let  $\S$ be a nonempty subset of $Q(D)$. We let $\S_\infty$ be the set of prime divisors $V$ of the second kind of $D$ such that $V$ contains infinitely many incomparable rings in $\S$.  By Theorem~\ref{minimal V}, $\S_\infty$ is precisely the set of patch limit points of $\S$ that are prime divisors of $D$ of the second kind.

 \end{notation}


Theorem~\ref{Zariski closure}     
implies that the Zariski closure of a  nonempty subset  $\S$  of $Q(D)$ 
can be calculated from how $\S$ is situated in the partially ordered set $Q^*(D)$. 
 
  \begin{theorem}  \label{Zariski closure} If $\S$ is a nonempty subset of $Q(D)$, then 
  the Zariski closure  of $\S$ in   $Q(D)$ is $ \downarrow \hspace{-.03in} (\S \cup  \S_\infty)$.
    \end{theorem}

\begin{proof}  
  By Remark~\ref{recover} the Zariski closure of $\S$  in $Q^*(D)$ is the downset of the patch closure $\overline{\S}$ of $\S$ in $Q^*(D)$, and so $\alpha \in Q(D)$ is in the Zariski closure of $\S$ in $Q(D)$ if and only if $\alpha$ is  in the downset of  $\overline{\S}$. Therefore to prove the theorem it suffices to show that $\downarrow \hspace{-.03in} \overline{\S} = \: \downarrow \hspace{-.03in} (\S \cup  \S_\infty)$. 
  By Theorem~\ref{minimal V}.2, $\S \cup \S_\infty \subseteq \overline{\S}$, and so 
  $ \downarrow \hspace{-.03in} (\S \cup  \S_\infty)\subseteq \:
  \downarrow \hspace{-.03in} \overline{\S}$.  
   To prove the reverse inclusion, 
  let $\alpha \in  \: \downarrow \hspace{-.03in} \overline{\S}$. Then there exists $R \in \overline{S}$ such that $\alpha \subseteq R$.   
     Since $\overline{\S}$ is the  union of  $\S$ and the set of patch limit points of $\S$ in $Q^*(D)$, we have either $R \in \S$, and hence $\alpha \in  \:  \downarrow \hspace{-.03in} \S$,  or $R$ is a patch limit point of $\S$ in $Q^*(D)$.  
     In the former case, we have $\alpha \in \: \downarrow \hspace{-.03in} (\S \cup  \S_\infty)$. To see that this is also true in the latter case, 
     suppose $R$   is a patch limit point of $\S$ in $Q^*(D)$. 
   By Theorem~\ref{patch closure}, $R$ is a valuation ring. 
   If $R \in X(D)$, then by Theorem~\ref{minimal V}.1 the fact that $\alpha$ is a subring of  $R$ implies that the local ring $\alpha$ is dominated by some member of $\S$. In this case, $\alpha \in \: \downarrow \hspace{-.03in} \S$.  
   On the other hand, if $R \not \in X(D)$, then  Theorem~\ref{minimal V}.2  
 implies $R$ contains infinitely many incomparable  rings in $\S$  and hence $R \in \S_\infty$.  Thus $\alpha  
 \in \: \downarrow \hspace{-.03in} (\S \cup  \S_\infty)$. This proves that   $\downarrow \hspace{-.03in} \overline{\S}   \subseteq \: \downarrow \hspace{-.03in} (\S \cup  \S_\infty)$, which completes the proof.  
\end{proof}

\begin{remark} A downset $\S$ in $Q(D)$ need not be Zariski closed. For example, let $\S =  Q_1(D) \cup \{D\}$. By Theorem~\ref{minimal V}, the order valuation ring $V$ of $D$ is the only 
 patch limit point of $\S$.  Therefore, by Theorem~\ref{Zariski closure}, the Zariski closure $\overline{\S}$ of $\S$ in $Q(D)$ is the set $P(D)$ of all points in $Q(D)$ that are proximate to $D$. Choose $W$ to be any rank two valuation overring of $D$ contained in $V$. By Remark~\ref{note2.1}.8.b there is an infinite chain of rings in $Q(D)$ contained in $W$ and hence in $\overline{\S}$. Only two of these rings are in $\S$, so  $\overline{\S} = P(D)$ is  larger than the downset $\S$.  
 \end{remark} 
 
 We focus next on the Noetherian subspaces of $Q^*(D)$. To simplify terminology, we use the following definition.

\begin{definition} A subset $\S$ of $Q^*(D)$ is {\it Noetherian} if it is a Noetherian space in the Zariski topology. \end{definition}


In Theorem~\ref{noether down} we characterize the Noetherian subspaces  of $Q^*(D)$. 
As a step in doing so, we show that the Zariski closure in $Q^*(D)$ of a  Noetherian subspace of $Q^*(D)$ is also a Noetherian space. Proving that  this is the case involves analyzing  the irreducible components of subsets of $Q(D)$ in terms  of downsets as defined in  Notation~\ref{downset}.

\begin{lemma} \label{irreducible}   
 A nonempty Zariski closed subset  $\S$ of $Q^*(D)$ is irreducible if and only 
 if $\S = \: \downarrow^* \hspace{-.02in} R$ for some $R \in Q^*(D)$.    Moreover, 
the following are equivalent for a nonempty    Zariski closed subset  $\S$ of $Q(D)$. 
\begin{enumerate}
\item
$\S$  is irreducible.

\item The Zariski closure of $\S$ in $Q^*(D)$ is irreducible. 

\item 
$\S = \: \downarrow \hspace{-.03in} R$ for some $R \in Q^*(D)$.\footnote{{If the 
set $\S$ is finite,  then $R \in Q(D)$.}}

\end{enumerate}

\end{lemma}

\begin{proof} The first assertion is a consequence of the fact that $Q^*(D)$ is a spectral space with respect to the Zariski topology (Corollary~\ref{Q* spectral}), and hence an irreducible closed set has a unique generic point. 

(1)  $\Rightarrow$ (2) 
Suppose  $\S$ is an irreducible closed  subset of $Q(D)$, and let $\overline{\S}$ denote the Zariski closure of ${\S}$ in $Q^*(D)$. To prove that $\overline{\S}$ is irreducible, it suffices to show that the Zariski closure of each nonempty Zariski open set in $\overline{\S}$ is $\overline{\S}$.  
Let $\U$ be a nonempty Zariski open set in $\overline{\S}$. Since $\S$ is dense in $\overline{\S}$, the set $\U \cap \S$ is nonempty. Thus, since $\S$ is irreducible,
 the Zariski closure of $\U \cap \S$ in $\S$ is $\S$.
Consequently, the Zariski closure of $\U$  in $Q^*(D)$ is $\overline{\S}$.

(2) $\Rightarrow$ (3) 
Suppose $\overline{S}$ is irreducible. The first assertion of the theorem shows that $\overline{\S} = \: \downarrow^* \hspace{-.04in} R$ for some $R \in Q^*(D)$.    
Now $\S = \overline{\S} \cap Q(D) = \: (\downarrow^* \hspace{-.04in} R) \cap Q(D) = \: \downarrow \hspace{-.02in} R$, which verifies (3).

 (3) $\Rightarrow$ (1)  Suppose that ${\S} = \: \downarrow \hspace{-.03in} R$ for some ring $R \in Q^*(D)$.  Let $\overline{(\:)}$ denote Zariski closure in $Q^*(D)$.   
We claim  that $R \in \: \overline{\S}$. If $R \in Q(D)$, then $R \in \: \downarrow \hspace{-.03in} R \subseteq \S$, and the claim is clear. Suppose $R \not \in Q(D)$. 
Then $R$ is a valuation ring, and to show that $R \in \: \overline{S}$, it suffices to prove that every Zariski basic open subset of $Q^*(D)$ that contains $R$ has nonempty intersection with $\S$.  Let  $x_1,\ldots,x_n $ be elements of the quotient field $F$ of $D$ such that  $R \in \U(x_1,\ldots,x_n)$. Then $x_1,\ldots,x_n \in R$.  Since $R$ is a valuation ring in $Q^*(D)$, there are 
 two cases to  consider, that in which $R$ is a minimal valuation ring  of $D$  and 
 that in which $R$ is a prime divisor of $D$ of the second kind. If $R$ is a minimal valuation 
 ring of $D$, then  $R$ is  a directed union of rings from $Q(D)$ by Remark~\ref{note2.1}.8.b. 
 Thus the set $ \S \cap \U(x_1,\ldots,x_n)$ is nonempty since $x_1,\ldots,x_n \in R$. 
 
Suppose that $R$ is a prime divisor of $D$ of the second kind. Let ${\ff M}$ denote the maximal ideal of $R$.  Then $D/\m_D  = (D + {\ff M})/{\ff M} \subsetneq R/{\ff M}$, and $R/{\ff M}$ 
 is a finitely generated field extension  $D/\m_D$ 
 of transcendence degree one. Therefore  $(D[x_1,\ldots,x_n] + {\ff M})/{\ff M}$ is a proper 
 subring of $R/{\ff M}$, so there exists a rank one valuation of $R/{\ff M}$ that contains
 $D/\m_D$.   By composite construction,  there exists a minimal 
  valuation ring $V$ of $D$ with 
  $V \subseteq R$ such that $x_1,\ldots,x_n \in V$.  Since $V$ is a directed union of rings from $Q(D)$, we conclude again that the set $ \S \cap \U(x_1,\ldots,x_n)$ is nonempty. 
  In all cases, $ \S \cap \U(x_1,\ldots,x_n)$ is nonempty, which proves that $R$ is in the Zariski closure of $\S$ in $Q^*(D)$.

To prove now that $\S$ is irreducible,  suppose $\T_1$ and $\T_2$ are Zariski closed subsets of $\S$ such that $\S = \T_1 \cup \T_2$, then $\overline{\S} =\overline{\T_1} \cup \overline{\T_2}$, where $\overline{(\:)}$  denotes  Zariski closure  in $Q^*(D)$. 
 We have proved that  $R \in \overline{\S}$, so  without loss of generality $R \in \overline{\T_1}$. Thus $\S = \: \downarrow \hspace{-.03in} R \subseteq \overline{\T_1}$. Since $\T_1$ is closed in $\S$, we have then that $\S \subseteq \S \cap \overline{\T_1} = \T_1$, proving that $\S = \T_1$. This shows that $\S$ is irreducible. 
\end{proof} 


A consequence of 
Theorem~\ref{noether down}  is that whether a subset $\S$ of $Q(D)$ is Noetherian can 
be detected from order-theoretic properties of $Q^*(D)$.

\begin{theorem} \label{noether down}  The following are equivalent for a nonempty subset $\S$ of $Q^*(D)$.
\begin{enumerate}
\item $\S$  is Noetherian.

\item $\S$ has only finitely many irreducible components. 

\item There  are valuation rings 
 $V_1,\ldots,V_n$ in $ Q^*(D)$ such that  $\S \subseteq \: \downarrow^* \hspace{-.04in} \{V_1,\ldots,V_n\}$.  
 
 \item The Zariski closure of $\S$ in $Q^*(D)$ is Noetherian. 
 \end{enumerate}
 If also $\S \subseteq Q(D)$, then items 1--4 are equivalent to 
 \begin{enumerate}

 \item[{\em 5.}]  The Zariski closure of $\S$ in $Q(D)$ is Noetherian.
 \end{enumerate}
\end{theorem} 

\begin{proof}  
(1) $\Rightarrow$ (2) This is clear since a Noetherian space has finitely many irreducible components \cite[Tag 0052]{stacks-project}.

(2) $\Rightarrow$ (3) 
Let $\S_1,\ldots,\S_n$ be the irreducible components of $\S$.  Let $\overline{(\ )}$ denote Zariski closure in $Q^*(D)$.   By Lemma~\ref{irreducible}, there exist rings $R_1,\ldots,R_n$ in $Q^*(D)$ such that $\overline{\S}_i =  \: \downarrow^* \hspace{-.04in} R_i$ for each $i$.  Thus $\S \subseteq \: \downarrow \hspace{-.03in} \{R_1,\ldots,R_n\}$.  For each $i$, let $V_i$ be a valuation ring in $Q^*(D)$ that dominates $R_i$.  Then  $\S \subseteq \:  \downarrow \hspace{-.03in} \{V_1,\ldots,V_n\}$, which verifies item 3. 

(3) $\Rightarrow$ (4) 
Suppose there are valuation rings 
 $V_1,\ldots,V_n$ in $ Q^*(D)$ such that  $\S \subseteq  \: \downarrow^* \hspace{-.04in} \{V_1,\ldots,V_n\}$.  Then the Zariski closure of $\S$ in $Q^*(D)$ is also a subset of the  Zariski closed set $\downarrow^*\hspace{-.04in} \{V_1,\ldots,V_n\}$. Since
 a subspace of  a Noetherian space is Noetherian, it suffices to prove that $\downarrow^* \hspace{-.04in} \{V_1,\ldots,V_n\}$ is Noetherian. Moreover, since  a finite union of Noetherian subspaces is Noetherian,  it suffices to show that every set of the form $C =  \: \downarrow^* \hspace{-.03in}V$, where $V$ is a valuation ring in $Q^*(D)$, is Noetherian.

Consider first the case that $V$ is a minimal valuation ring of $D$.  Then 
$\downarrow \hspace{-.03in}  V$ consists of the quadratic sequence $D = \alpha_0 \subseteq \alpha_1 \subseteq \cdots $ along $V$. The proper Zariski closed subsets of $C$ are precisely the sets $\downarrow \hspace{-.03in}  \alpha_i$. Since these sets are finite, the closed sets of $C$ satisfy the descending chain condition. Therefore $C$ is a Noetherian space in the Zariski topology. 

Next suppose that $V$ is not a minimal valuation ring of $D$. 
To deal with this case, we prove first the following claim.

\smallskip

{\textsc{Claim:}} If $V$ is not a minimal valuation ring, then no proper Zariski closed subset 
 of $C= \: \downarrow^* \hspace{-.03in}  V$ contains infinitely many incomparable rings. 

\smallskip

Let $B$ be a  Zariski closed set of $C$ that contains infinitely many incomparable rings.  
We prove $B = C$.  Either $B$ contains infinitely many incomparable rings in $Q(D)$ or $B$ contains infinitely many valuation rings.  
 By Theorem~\ref{minimal V}.2, $V$ is a prime divisor of $D$ of the second kind and 
 $V$ is a patch limit point of every infinite set of incomparable rings in 
 $\downarrow \hspace{-.03in}  V$. If $B$ contains an infinite set of incomparable 
  rings in $Q(D)$, then $V$ is a patch limit point of $B$. Since $B$ is Zariski closed 
  in $Q^*(D)$, Remark~\ref{recover} implies $B = C$
  and the claim is proved.  
 
 To complete the proof of the claim, it remains to consider the
 case where $B$ contains an infinite set $\S$ of valuation rings in $Q^*(D)$. We may assume that each valuation ring in $\S$ is a proper subring of $V$, since otherwise the claim is clear from the fact that $C= \: \downarrow^* \hspace{-.04in}  V$.  
 Since $B$ is a Zariski closed subset of $C$ containing $\S$, to show that $B = C$ it suffices to show that the Zariski closure of $\S$ is $C$. 
Let $\U$ be a nonempty open subset of $C$.  Since $C = \: \downarrow^* \hspace{-.04in}  V$, 
there exist $x_1,\ldots,x_n \in V$ such that $\U = C \cap \U(x_1,\ldots,x_n)$.   
 %
  Let ${\ff M}$ denote the maximal ideal of $V$.  Since $V$ is a prime divisor of the 
  second kind of $D$,  
  $D/\m_D =   (D + {\ff M})/{\ff M} \subsetneq V/{\ff M}$,  and $V/{\ff M}$ 
   is a finitely generated field extension of transcendence degree one of $D/\m_D$.  
   It follows that each  nonzero element in the field $V/{\ff M}$ is contained 
   in all but at most finitely many of the rank one valuation rings of  the field $V/{\ff M}$ 
   that contain $(D + {\ff M})/{\ff M}$.   
     For each valuation ring $W$ in $\S$, we have  $(D + {\ff M})/{\ff M} \subseteq W/{\ff M} \subseteq  V/{\ff M}$ and $W/{\ff M}$ is a rank one valuation ring of the field $V/{\ff M}$.  
     Thus there exists a valuation ring $W \in \S$ such that $x_1,\ldots,x_n \in W$. Therefore  $W \in \S \cap \U(x_1,\ldots,x_n)$, which proves that $\S$ is dense in $C$ in the Zariski topology. Hence $B = C$, and this completes the proof of the claim. 
     
     \smallskip

We use the claim  to prove that $C=  \: \downarrow^* \hspace{-.03in}  V$  is a Noetherian space. It suffices to show that 
 every proper Zariski closed subset $B$ of $C$ is a Noetherian space. 
 If $B$ is finite, this is clear, so suppose $B$ is infinite. 
Since $B$ is proper and $V$ is the generic point of $C$, we cannot have $V \in B$.   
As a closed subset of the spectral space $Q^*(D)$ (with respect to the Zariski topology),   $B$ has minimal elements with respect to the partial order $R \leq S$ iff $S$ is in the Zariski closure of $\{R\}$.  The elements in $B$ minimal with respect to this partial order are precisely the rings that are maximal in $B$ with respect to set inclusion. Thus every ring in $B$ is contained in a ring in $B$ that is maximal with respect to set inclusion. Since the rings in $B$ that are maximal with respect to set inclusion are incomparable, the claim implies that there are only finitely many of them, say $R_1,\ldots,R_n$.  Therefore  $B = \: \downarrow^* \hspace{-.04in} \{R_1,\ldots,R_n\}$.  
 We have already established that each Zariski closed set $\downarrow^* \hspace{-.04in} R_i$ is a Noetherian space  since as a proper subring of $V$  each $R_i$ is contained in a minimal valuation ring. Thus $B$ is a Noetherian space since it is a
finite  union of Noetherian subspaces. This proves that every proper 
 closed subset of $C$ is Noetherian. It follows that $C$ satisfies the descending chain condition on closed sets and hence $C$ is Noetherian. This proves that item 3 implies item 4.

That item 4 implies items 1 and 5 
follows from the fact that a subspace of a Noetherian space is Noetherian. If $\S$ is a subset of $Q(D)$, item 5 similarly implies item 1. 
 \end{proof}

\begin{remark} Noetherian subspaces of the Zariski-Riemann space of valuation overrings of a two-dimensional Noetherian domain are the subject of {\cite{ONoeth, OInter}}. See also~{\cite{ONoeth2, OTop2}}. 
\end{remark}

 \begin{remark} \label{commutes} If $\S$ is a nonempty Noetherian subset of $Q^*(D)$   with respect to the Zariski topology and  $R$ is a flat $D$-submodule of $F$, then $(\bigcap_{\alpha \in \S}\alpha)R = \bigcap_{\alpha \in \S}(\alpha R ).$
 This is because flat submodules of the quotient field of  a domain  commute with intersections  of rings from 
quasicompact sets of overrings \cite[Theorem~3]{FS}.   
\end{remark}






Examples presented in Theorems \ref{thm6.5} and \ref{thm6.4} in Section~\ref{sec6}
  show that even for Noetherian subsets $\U$ of $Q^*(D)$, $\O_\U$ need not be a Noetherian ring. We consider in Lemma~\ref{clopen2} and Theorem~\ref{locally closed} restrictions on $\U$ that guarantee $\O_\U$ is a Noetherian ring.

\begin{lemma} \label{clopen2} If a subset $\S$ of $Q^*(D)$ is an intersection of a downset and a Zariski quasicompact open set  in $Q^*(D)$, then $\O_\S$ is a normal Noetherian  domain. 
\end{lemma}  

\begin{proof} Write $\S = \U \cap \V$, where $\U$ is a quasicompact open set in $Q^*(D)$ and $\V$ is a downset in $Q^*(D)$.  
By Corollaries~\ref{basis} and~\ref{Q* spectral}, the quasicompactness of $\U$  implies there exist finite subsets  $A_1,\ldots,A_n$ of $F$ such that $\S = ({\mathcal{U}}(A_1) \cup \cdots \cup {\mathcal{U}}(A_n)) \cap \V $.  Let $\T$ be the set of rings that are minimal with respect to inclusion in $\S$.  Since  $\V$ is a downset, the rings in $\T$ are minimal in the set ${\mathcal{U}}(A_1) \cup \cdots \cup {\mathcal{U}}(A_n)$. Thus $\O_\T = \O_\S$.  The 
 intersection of finitely many normal Noetherian overrings of $D$ is a Krull domain and hence by  Remark~\ref{disc2.7} is  a normal Noetherian ring. Therefore to verify the theorem it suffices to  prove that $\O_{\S}$ is a Noetherian ring in the case where $\S = \U(A_1)$.  
 
 Assuming $\S = \U(A_1)$, 
the rings in $\T$ are minimal in $Q(D)$ with respect to containing $D[A_1]$. Write $A_1 = \{x_1,\ldots,x_t\}$.  Then $D[A_1]$ is the   coordinate ring of an affine component of the projective model $X$ defined by $x_1,\ldots,x_t$.  Since the rings in $\T$ are elements of $Q(D)$ that are minimal with respect to dominating closed points in $X$, these points are also closed points in the minimal desingularization $X'$ of $X$ \cite[Theorem 5.3]{HO}. This proves that $\T$ is a set of closed points in a nonsingular projective model over $D$.  By \cite[Theorem 7.4]{HO}, $\O_\T$ is a Noetherian normal ring.  
\end{proof}

\begin{theorem} \label{locally closed}  Let $\V$ be a  Noetherian downset in $Q^*(D)$.  Then for every Zariski open subset $\U$ of $\V$, the ring $\O_\U$ is a Noetherian normal domain.

\end{theorem} 

\begin{proof} Let $\U$ be a Zariski open subset of $\V$, and let $\U'$ be a Zariski open subset of $Q^*(D)$ such that $\U = \U' \cap \V$.  By Corollary~\ref{Q* spectral}, $Q^*(D)$ is a spectral space, so   $\U'$ is a union of Zariski quasicompact open subsets of $Q^*(D)$, say $\U' = \bigcup_i \U_i$, where each $\U_i$ is a Zariski quasicompact open subset of $Q^*(D)$.   Then $\{\V \cap \U_i\}$ is an 
open cover of $\U$. Since $\V$ is a Noetherian subset of $Q^*(D)$, every subset of $\V$ is quasicompact in the Zariski topology, and so $\U$ has a finite subcover in $\{\V \cap \U_i\}$. Since the union of {finitely many} quasicompact subspaces is quasicompact, we conclude that $\U = \V \cap {{\mathcal{W}}}$ {for some quasicompact open subset ${\mathcal{W}}$ of $Q^*(D)$}. By Lemma~\ref{clopen2}, $\O_\U$ is a Noetherian normal domain. 
\end{proof}

\begin{corollary} \label{P} Let $\alpha \in Q(D)$. The subset $P(\alpha)$ of $Q(D)$ of points proximate to $\alpha$ is a Noetherian subset of $Q(D)$ such that for each Zariski open subset $\U$ of $P(\alpha)$, $\O_\U$ is a Noetherian normal domain. 
\end{corollary}

\begin{proof}
If $V$ is the order valuation ring of $\alpha$, then $P(\alpha) \subseteq \: \downarrow \hspace{-.03in} V$, so $P(\alpha)$ is Noetherian by Theorem~\ref{noether down}. Since $P(\alpha)$ is a downset of $Q^*(\alpha)$, 
the second claim now follows from Theorem~\ref{locally closed} applied to  $Q^*(\alpha)$. 
\end{proof}




\begin{remark} \label{presheaf}  We may view $\O$   as a presheaf on $Q^*(D)$ with respect to either the patch or Zariski topologies. In either topology, the stalks of the presheaf $\O$ are the rings in $Q^*(D)$. 
\begin{enumerate} 
\item Working with the patch topology on $Q^*(D)$, we have by  Lemma~\ref{clopen2} that for each patch clopen set $\U$, the ring of sections $\O_\U$ of $\U$ is Noetherian.   
 
 \item 
Since $\O$ is defined on the empty set to be the field $F$,  $\O$ is not a sheaf with respect to either the patch or Zariski topologies of $Q^*(D)$ because for $\O$ to be a sheaf would require $\O$  to be the zero ring on the empty set. Even with the modification that $\O$ is defined on the empty set to be the zero ring, $\O$ is not a sheaf. This follows from the fact that $Q^*(D)$ is not an irreducible space in either the Zariski or patch topologies.

\item 
By restricting to $P(\alpha)$, where $\alpha \in Q(D)$, we obtain that the  modification to the definition of $\O$ as discussed in item 2 produces a sheaf on $P(\alpha)$ with respect to the Zariski topology, since by Lemma~\ref{irreducible}, $P(\alpha)$ is  irreducible in the Zariski topology. Therefore  $P(\alpha)$ with structure sheaf $\O$ is a locally ringed Noetherian space for which the ring of sections of each open set is a Noetherian ring. By appending the order valuation of $\alpha$ to the set $P(\alpha)$, we obtain a locally ringed spectral Noetherian space whose rings of sections are Noetherian. This locally ringed space is not a scheme since the stalks of the structure sheaf have localizations that do not appear as other stalks.


  \end{enumerate}
 \end{remark}



 

\section{Examples}  \label{sec6}

The purpose of this section is to present  examples that illustrate the  
range of behavior of intersections 
of {rings in} subsets of $Q(D)$.  We use the following terminology.

\begin{definition}  \label{def5.1}    A subset $\U$ of $Q(D)$ is said to be {\it complete} 
if $\O_\U = \bigcap_{R \in \U}R = D$.  For a point $\alpha \in Q(D)$, a subset $\U$ of $Q(D)$ of points
that dominate $\alpha$ is said to be {\it complete over} $\alpha$ if $\O_\U = \alpha$. 
\end{definition}


\begin{remark} \label{rem5.3} 
Let  $J$ be an integrally closed  ideal of $D$.  
The  set $\U$ of closed points of the projective model $X = \Proj D[Jt]$ over $D$ is complete over $D$,  since every valuation ring $V \in X(D)$ 
is centered on a closed point of $X$, and $D$ is the intersection of the valuation rings in $X(D)$.
We are identifying $\Proj D[Jt]$ with the set of local rings obtained by homogeneous localization of 
the graded domain $D[Jt]$ at its relevant homogeneous prime ideals.
Therefore  $D = \bigcap_{\alpha \in \U}\alpha$. 
It is natural to ask if $D$ is the intersection of the rings in  a proper subset of $\U$. 
 This is equivalent to asking if the representation $D = \bigcap_{\alpha \in \U}\alpha$
is irredundant.  {If $D$ is Henselian and $X$ is nonsingular, then this representation is irredundant; see Remark~\ref{rm3.3}.2.}
\end{remark}



\begin{example}  \label{ex1.7}
Assume Notation~\ref{note3.1}.  
There exist 2-dimensional normal Noetherian local overrings $R$ and $S$ of $D$
such that $D = R ~ \cap ~ S$ and  $R$ and $S$    both  properly contain $D$.  
Let  $R = D[y^2/x]_{(x, y, y^2/x)D[y^2/x]}$  and  $S = D[x^2/y]_{(x, y, x^2/y)D[x^2/y]}$. 
Then $R$ and $S$ both properly contain $D$. Since $D$ is the intersection of valuation domains $V$
centered on $\m_D$,  to prove that $D = R~ \cap ~S$, it suffices to prove that each such valuation domain 
$V$ has the property that either $R \subseteq V$ or $S \subseteq V$.  Since $V$ has center $\m_D$ 
on $D$, both $xV$ and $yV$ are contained in $\m_V$.  If $y^2V \subsetneq xV$,  then $R \subseteq V$.
On the other hand, if $xV \subseteq y^2V$,  then $x^2V \subsetneq xV \subseteq y^2V \subsetneq yV$
implies that $x^2V \subsetneq yV$ and $S \subseteq V$.   Therefore $D = R~ \cap ~ S$.
\end{example}

 \begin{discussion} \label{dis4.3}
 Assume Notation~\ref{note3.1}  and let $\gamma \in Q_1(D)$.  
 It  observed in  \cite[Theorem~8.3]{HO} that  the set  
 $Q_1(D)  \setminus \{\gamma\}$ is not
  complete.  Let $Q_1(\gamma)$ denote the points in $Q_2(D)$ that dominate 
  $\gamma$.\footnote{This is 
  precisely  the infinitely near  points in the first neighborhood of $\gamma$.}
  Then $\gamma$ is the intersection of the points in $Q_1(\gamma)$.
  Hence $\U = Q_1(\gamma ) \cup (Q_1(D) \setminus \{\gamma\})$ is a 
  complete subset of $Q(D)$. 
  We are interested in describing subsets $\Lambda$ of $Q_1(\gamma)$ such that 
  $\Lambda \cup (Q_1(D) \setminus \{\gamma\})$ is complete.

   \end{discussion}

 \begin{example} \label{ex4.6}
   Assume Notation~\ref{note3.1} and  assume that $D$ has an algebraically closed 
   coefficient field $k$.  Let 
   $$
   \beta~ = ~    D[\frac{x}{y}]_{(y, \frac{x}{y})D\frac{[x}{y}]}   ~ \in   ~Q_1(D), 
   $$
      and consider the nonsingular projective 
   model  $X$ over $\Spec D$ gotten by blowing up $\m_D$ and then blowing up $\beta$. 
   The Zariski theory   \cite[Appendix 5]{ZS2} or \cite{Hun}  implies that 
    $X = \Proj D[Jt]$,  where $J = (x, y)\dot (x, y^2) = (x^2, xy, y^3)D$.  
   The model $X$ is obtained from  $\Proj D[xt, yt]$ by removing the point $\beta$ and adding the 
   projective line   $Q_1(\beta)$.  The set of closed points of $\Proj D[Jt]$ is 
   $$  \U ~  :=  ~  (Q_1(D) \setminus \{\beta\}) \cup Q_1(\beta).$$
   As in Discussion~\ref{dis4.3},
the set $\U$ is complete and so
    $\O_\U = D$.   To show this representation is irredundant,  we use:

 \begin{fact}  \label{fact6.6}
    To show  an element $\delta  \in \U$ is irredundant in 
    the representation $D = \bigcap_{R \in \U}R$,  it suffices to show 
     that there exists an essential valuation ring $V$ of $D$ such that $\delta \subset V$ and no other element of $\U$ 
     is contained in $V$ \cite[Lemma 8.5]{HO}.
 \end{fact}
    
  The points in $Q_1(D)  \setminus \{\beta\}$ are the localizations of $D[\frac{y}{x}]$ at maximal ideals containing
  $xD[\frac{y}{x}]$.  Since $k$ is an  algebraically closed  coefficient field for $D$,  these 
  maximal ideals have the form
  $$
  \m_b     ~=~ 
  (x, ~ \frac{y}{x} - b)D[\frac{y}{x}],  \quad \text{ where} \quad b \in k.
  $$
  For each $b \in k$,  let $\p_b = (y - bx)D$,  let $V_b = D_{\p_b}$,  and let 
  $R_b = D[\frac{y}{x}]_{\m_b}$.  
  Then the set  $\{R_b\}_{b \in k}   = Q_1(D)  \setminus \{\beta\}$ and for each
   $b \in k$,  we have $R_b \subset V_b$,  and $R_b$ is the unique point in $Q_1(D)$ that 
  is contained in $V_b$.    Since $\beta$ is not contained in $V_b$,  no point of $Q_1(\beta)$ is
  contained in $V_b$.  Hence $R_b$ is the unique point of $\U$  that is contained in $V_b$.
  
  Let $x_1 = \frac{x}{y}$.   Then $(x_1, y)\beta$ is the maximal ideal of $\beta$.
  The points in $Q_1(\beta)$ are the localizations of $\beta[\frac{x_1}{y}]$ at  maximal ideals
  containing $y\beta[\frac{x_1}{y}]$,   and the localization of $\beta[\frac{y}{x_1}]$ at the  maximal 
  ideal  $(x_1,   \frac{y}{x_1})\beta[\frac{y}{x_1}]$.  Let $\gamma$ denote the localization 
   of $\beta[\frac{y}{x_1}]$ at the  maximal ideal    $(x_1,   \frac{y}{x_1})\beta[\frac{y}{x_1}]$. 

  The maximal ideals of $\beta[\frac{x_1}{y}]$  containing $y$  have the form 
  $$
  \n_b ~ = ~  (y, ~   \frac{x_1}{y} - b)\beta[\frac{x_1}{y}], \quad \text{ where } \quad b \in k.
  $$
  For each $b \in k$,  let $\q_b = (x - by^2)D$, let  $W_b = D_{\q_b}$, and 
   let $S_b =  \beta[\frac{x_1}{y}]_{\n_b}$.
  Then the set $\{S_b\}_{b \in k} =  Q_1(\beta) \setminus \{\gamma\}$, 
  and for each $b \in k$,  we have $S_b \subset W_b$.    
  Since $D/\q_b$ is a DVR,  $S_b$ is the unique point of $\U$ contained in $W_b$.  
  
  To show the representation $D = \O_U$ is irredundant,  by Fact \ref{fact6.6}, it remains to show there exists an
  essential valuation ring $V$ of $D$ such that $\gamma \subset V$ and $\gamma$ is the only point of $\U$
  contained in $V$. Let $\p = (x^2 - y^3)D$ and let $V = D_\p$.  Since the integral closure of $D/\p$ is local,
  $V$ contains at most one point of $\U$.  
  
  Since $x = yx_1$,   the transform of $\p$ in $D[x_1]$ is $ (x_1^2 - y)D[x_1]$. 
  Let $y_1 = \frac{y}{x_1}$,  then $y = x_1y_1$ and the transform of  $ (x_1^2 - y)D[x_1]$ in $\beta[y_1]$ 
  is $(x_1 - y_1)\beta[y_1]$.  Since $(x_1, y_1)\gamma = \m_{\gamma}$,  we have
  $\gamma \subset V$.  Therefore the representation $D = \O_\U$ is irredundant.

    
    
  \end{example}
  
\begin{remark}
With the notation of Example~\ref{ex4.6}, and so $y_1 = \frac{y^2}{x}$,  
 it is shown in \cite[Example 7.6]{HO} that $ R =   D[y_1]_{(x, y, y_1)D[y_1]}$ is a 2-dimensional normal Noetherian local 
 domain, and $\Proj D[Jt]$  is the minimal desingularization of $R$.
\end{remark}

Theorem~\ref{th1.8}  establishes  that every minimal valuation overring of $D$ is a 
localization of a  ring in $\R(D)$.
Theorem~\ref{thm6.4} and Theorem~\ref{th1.8}  both describe  
subsets $\U$ of $Q(D)$ for which the rings in $\U$ are incomparable  and 
 $\O_\U$ is 
not an almost Krull domain.

 \begin{theorem}   \label{th1.8}
 Let $V$ be a minimal valuation overring of $D$.   There exists a subset $\U$  of $Q(D)$ such  
 that that  $V$ is a localization of $C = \O_\U$.  In particular,  if $V$ is
 chosen not to be a DVR,  then $C$ is not an  almost Krull domain.  
 \end{theorem}
 
 \begin{proof}  
Let 
$D = \alpha_0 \subset \alpha_1 \subset \alpha_2 \subset   \cdots $  be the quadratic sequence 
determined by $V$ {as in Remark~\ref{note2.1}.8}.  
For each  integer $i \ge 1$  let $ \beta_i$   be a first level quadratic extension of $\alpha_i$, different from $\alpha_{i+1}$,  and let 
$C = \bigcap_{i=1}^\infty \beta_i$. Thus $C = \O_\U$,  where $\U = \{\beta_i\}_{i=1}^\infty$. 
  We observe that $C$ is the following directed union.   
Since $\alpha_i \subset \beta_i$ for each $i$,  we have
$$
\alpha_1 ~  \subseteq (\beta_1 \cap \alpha_2)   ~ \subseteq (\beta_1 \cap \beta_2 \cap \alpha_3) ~ \subseteq ~ \cdots ~ \subseteq C.
$$

For each integer $n \ge 1$,  let  $ C_n   := (\bigcap_{i=1}^n\beta_i) \cap \alpha_{n+1}$ .    Since 
$\alpha_{n+1}   \subseteq \bigcap _{ j > n} \beta_j$,   we have  
$\bigcup_{n = 1}^\infty C_n = C$.    Corollary~6.5 of \cite{HO}  implies that   $C_n$
is a Noetherian regular domain  with precisely $n+1$ maximal ideals all of height 2,  
and the localizations of 
$C_n$ at its maximal ideals are $\beta_1, \ldots, \beta_n$ and $\alpha_{n+1}$.

Let $\m_V$ denote the maximal ideal of $V$.  Then $\m_V \cap \alpha_i = \m_i$  is  
the maximal ideal of $\alpha_i$,
and $\m_V = \bigcup_{i=1}^\infty \m_i$.  

Let $\p_n = \m_V \cap C_n$.   Then $(C_n)_{\p_n} = \alpha_{n+1}$.  Since this holds for all $n$,
we have $V = \bigcup \alpha_n = \bigcup (C_n)_{\p_n}$.  Let $\p   = \m_V \cap C$ denote the center of 
$V$ on $C$.  It follows that $C_\p = V$.   If $V$ is not  a DVR,  then  $C$ is not an almost Krull domain.  
\end{proof}

We observe in Remark~\ref{rmk6.14} that the set $ \U$ in 
Theorem~\ref{th1.8} is not Noetherian.

\begin{remark} \label{rmk6.14}  
Let   $\U = \{\beta_i\}_{i=1}^\infty$,       and   the valuation ring $V = \bigcup_{n \ge 0}\alpha_n$  
be   as in Theorem~\ref{th1.8}.  
Each $\beta_n$ is a point distinct from $\alpha_{n + 1}$ 
in the first neighborhood $Q_1(\alpha_n)$ of $\alpha_n$.      Since $V$ dominates 
$\alpha_{n+1}$,  $V$ does not dominate $\beta_n$.

 Each $\beta_i$ is a maximal 
element of $\U$ and   a maximal element of the down set $\downarrow \hspace{-.03in} \U$. By
construction,  $\alpha_n \in   ~\downarrow \hspace{-.03in} \U$ for all $n $. 
Theorem~\ref{minimal V}.1 implies that $V$ is a patch limit point of $\U$.

The rings $\beta_n$ are incomparable,  that is $\beta_i \subseteq \beta_j$ implies $i = j$. 
 If $W$ is a 
minimal valuation overring of $D$,  then at most one of the $\beta_n$ is dominated by $W$. 

Notice that $\downarrow \hspace{-.03in} \U   = \U \cup \{\alpha_n\}_{n = 0}^\infty$.
Let $W \in X(D)$ with $W \ne V$.  Then only finitely many of the $\alpha_n$ are contained in $W$. 
Since the $\beta_i$ are incomparable, no more than one of the $\beta_i$ is dominated by
 $W$.  Since $W \in  X(D)$,   for each integer $n \ge 1$ there is a unique ring $\gamma_n \in Q_n(D)$
 that is contained in $W$ and $W$  then  dominates $\gamma_n$.  
 Therefore no more than one of the $\beta_i$ is contained in $W$.  
 By Theorem~\ref{minimal V}.1,  $W$ is not a patch limit point of $\U$.  
 
 Let $W$ be a prime divisor of the second kind on $D$, and let $\gamma \in Q(D)$ be such
 that $W = \ord_{\gamma}$.  We consider two cases:
 \begin{enumerate}
 \item 
 Assume  $\gamma$ is contained in $V$.  Then $\gamma$ is dominated by $V$ 
 and  $\gamma = \alpha_n$ for some $n$. Since $Q_1(\alpha_n) \cap  \downarrow \hspace{-.03in} \U $
 is finite,  Theorem~\ref{minimal V}.2 implies that $W$ is not a patch limit point of $\U$.  
 
 \item
 Assume $\gamma$ is not contained in $V$.  Then $\alpha_n$ does not dominate $\gamma$, 
 and $\beta_n \in Q_1(\alpha_n)$ implies $\beta_n$ does not dominate $\gamma$.
 Hence $\downarrow \hspace{-.03in} \U \cap Q_1(\gamma)$ is empty  and by
 Theorem~\ref{minimal V}.2,  $W$ is not a patch limit point of $\U$. 
 \end{enumerate}
 
 We conclude that $V$ is the unique patch limit point of $\U$.  Since the $\beta_n$ are 
 maximal elements of $\U$ and are not in $V$,  the set $\U$ has infinitely many  irreducible 
 components and is not Noetherian.


\end{remark}

\begin{remark} \label{rmk6.15}  {Corollary~\ref{cor4.10} implies that every  infinite subset of $Q(D)$ has at least one 
patch limit point in $Q^*(D)$.    Remark~\ref{rmk6.14} illustrates that a subset $\U$ of
$Q(D)$ with one patch limit point may also have infinitely many finite irreducible components.}
\end{remark}

Example~\ref{disc6.3} is a  $2$-dimensional normal local overring $S$  of $D$  that is 
dominated by a local ring in $Q_1(D)$, but $S$ is  not in $\R(D)$.

  \begin{example} \label{disc6.3}
  Assume Notation~\ref{note3.1}. With $k$ an algebraically closed field of characteristic zero and $D$ the 
  localized polynomial ring $k[x, y]_{(x,y)}$,  Shannon in \cite[Example 3.9, p.306]{S}  
  describes an infinite strictly ascending chain 
  $$
  D ~ \subset  ~R_2 ~ \subset  ~R_3 ~ \subset \cdots \subset  ~ R_n ~ \subset \cdots 
  $$ 
  of 2-dimensional normal Noetherian  local   domains,  where 
  $$R_2 = D[y^2/x]_{(x, y, y^2/x)D[y^2/x]},$$
  $R_3$ is the localization of $D[y^2/x, y^3/x^2]$ at the maximal ideal 
  $$
   (x, y, y^2/x, y^3/x^2)D[y^2/x, y^3/x^2],
   $$
  and $R_n$ is the localization of $D[y^2/x, y^3/x^2, \ldots, y^n/x^{n-1}]$ at the maximal ideal generated by 
   $(x, y, y^2/x, \ldots ,  y^n/x^{n-1})$, for $n \ge 3$.  
    Each of the rings $R_n$ is dominated by   $\alpha = D[y/x]_{(x, y/x)D[y/x]} \in Q_1(D)$.   

Since each of the rings $R_n$ is a 2-dimensional normal Noetherian local domain, 
each $R_n \in \R(D)$  by \cite[Theorem 7.4]{HO}.
  Let  $S = \bigcup_{n =2}^\infty R_n$.
     Then $S$ is a 2-dimensional normal local domain
    dominated by $\alpha$.  Since $y/x \not\in S$,  we have $S \subsetneq \alpha$. 
    Since $ y(y/x)^n \in S$ for all $n \ge 1$,  the element $y/x$ is almost integral over $S$, and $S$ is not completely integrally closed.
     However, a ring in $\R(D)$  is an intersection of completely integrally closed domains and is therefore  
     completely integrally closed. This 
 implies that $S$ is not in $\R(D)$.  
    
  \end{example}

We use Setting~\ref{alanexample} in Theorems~\ref{thm6.5} and \ref{thm6.4}.

\begin{setting} \label{alanexample}
Assume Notation~\ref{note3.1}, and assume that $D$ has an algebraically
closed coefficient field $k$.  Let $\ord_D$ denote the order valuation ring of $D$,  and define
$$
 A: ~ = ~  D[\{\frac{y^2}{x + ay} ~|~ a \in k\} \cup \{\frac{x^2}{y}\}].
$$
Then $A \subset \ord_D$ and $\ord_D$ is centered on the  maximal ideal $\m$  of $A$, where 
$$
\m: ~=   (x, y, \{\frac{y^2}{x + ay} ~|~ a \in k\} \cup \{\frac{x^2}{y}\})A.
$$
Define  $C := A_\m$.  We prove in  Proposition~\ref{prop6.4}   that $C$ is an 
infinite directed union of $2$-dimensional normal Noetherian local domains. Therefore $C$ 
is integrally closed. 

For each $a \in k$, define
$$
\alpha_a :~=  ~ D[\frac{x + ay}{y}]_{(y, \frac{x+ ay}{y})}  \in Q_1(D) \quad 
\text{ and} \quad  \alpha':~ =~ D[\frac{y}{x}]_{(x, \frac{y}{x})} \in Q_1(D).
$$

For each $a \in k$, define
$$
\beta_a:~ =~ \alpha_a[\frac{y^2}{x + ay}]_{(\frac{x + ay}{y}, \frac{y^2}{x + ay})} \in Q_2(D) \quad  \text{  and   } \quad
\beta':~ =~ \alpha'[\frac{x^2}{y}]_{(\frac{y}{x}, \frac{x^2}{y})}  \in Q_2(D). 
$$
Let $B ~ = ~  \bigcap_{a \in k} \beta_a $, and let $C' = B \cap \beta'$
\end{setting}

\begin{proposition} \label{prop6.4}
The local ring $C$  defined in Setting~\ref{alanexample} is an infinite directed union of normal Noetherian local 
domains.
\end{proposition}
\begin{proof}
Let $n$ be a positive integer and let $\mathcal S =  \{a_1, \ldots, a_n\}$ be  a set of 
$n$ distinct elements of $k$.  
Define the ring $R_\S$ to be the localization of $D[\frac{x^2}{y},  \{\frac{y^2}{x + a_iy} \}_{i=1}^n]$
at the maximal ideal generated by $(x, y, \frac{x^2}{y},  \{\frac{y^2}{x + a_iy} \}_{i=1}^n)$.
Each of the ideals $(y, x^2)D$,  $(x+a_1y, y^2)D, \ldots , (x+a_ny, y^2)D$  is a simple complete ideal.  
Let $J$ denote  the product of these $n+1$ complete ideals.  Since $D$ is a 2-dimensional regular local ring, $J$ is a
complete ideal,  and  $R_\S$ is the local ring on $\Proj D[Jt]$ dominated by $\ord_D$.  Therefore $R_\S$ is a 
normal Noetherian local domain that is dominated by $C$. 

The rings $R_\S$ obtained as we vary over finite sets of distinct elements of $k$ form a directed family of normal 
Noetherian local domains dominated by $C$, and $C$ is the directed union of this family of rings. 
Remark~\ref{rmk6.5} implies that for $ \S_1 \subsetneq \S_2$,  the associated ring $R_{\S_1} \subsetneq 
R_{S_2}$.     \end{proof}

\begin{remark} \label{rmk6.5}
Let $R :=R_S$ be as defined in the proof of Proposition~\ref{prop6.4}. 
The minimal desingularization of $\Spec R$ is $\Proj R[\m_Rt]$.  All but $n+1$ of the 2-dimensional regular local 
rings in $\Proj R[\m_Rt]$ are local rings in $Q_1(D)$,  the elements in $Q_1(D)$  that are missing are 
 $\alpha' = D[\frac{y}{x}]_{(x, \frac{y}{x})} $ and  
  $\alpha_{a_i} = D[\frac{x + a_{i}y}{y}]_{(y, \frac{x+ a_{i}y}{y})} , ~1 \le i \le n$.
  
  The 2-dimensional regular local rings in $\Proj R[\m_Rt]$ that are not in $Q_1(D)$ are  in $Q_2(D)$ and
  are  $ \beta' = \alpha'[\frac{x^2}{y}]_{(\frac{y}{x}, \frac{x^2}{y})} $ and 
  $\beta_{a_i}     = \alpha_{a_i}[\frac{y^2}{x + a_{i}y}]_{(\frac{x + a_{i}y}{y}, \frac{y^2}{x + a_{i}y})},   ~1 \le i \le n$. 
  
\end{remark}

\begin{theorem} \label{thm6.5}
Assume  Setting~\ref{alanexample},  and let $\U = \{\beta_a ~| ~ a \in k  \}$. 
Then $B = \O_\U =  \bigcap_{a \in k}\beta_a$  is a non-Noetherian almost Krull domain, 
and $\U$ defines an irredundant essential representation of $B$.
\end{theorem}
\begin{proof}
Observe that $D[\frac{x}{y}] \subseteq  B$.
Let $a, b \in k$. Then $\frac{y^2}{x + ay} \in \beta_b \iff  y^2 \in (x + ay)\beta_b$.  If $a = b$ 
this follows because $\frac{y^2}{x + ay} \in \beta_a$.
Assume $a \ne b$.  To compute the transform of $x + ay$ in 
$\alpha_b = D[\frac{x}{y}]_{(y, \frac{x}{y} + b)}$,  let $x_1 = \frac{x}{y}$. Then $x = yx_1$  and
$ x + ay  = yx_1 + ay = y(x_1 + a)$.  Since $a \ne b$, $x_1 + a$ is a unit of $\alpha_b$ and hence 
also a unit of $\beta_b$.  Therefore $y^2 \in (x + ay)\beta_b$ also in this case.
It follows that $D[\frac{x}{y}][\{ \frac{y^2}{x + ay} ~|~a \in k \}] \subseteq B$.

Let $\q_a = (\frac{x + ay}{y})\beta_a$.  Then $\q_a$ is a height 1 prime of $\beta_a$ and 
 $(\beta_a)_{\q_a}$ is the order valuation ring of $\alpha_a$.
The image of $\frac{y^2}{x + ay}$  in the field of fractions  of $\beta_a/\q_a$ is transcendental over $k$. 
Hence  $\p_a  := \q_a \cap B$ is a height 1 prime ideal of $B$ and  $ B_{\p_a} = (\beta_a)_{\q_a}$. 

Since $\p_a \cap D = \m_D$ for every $a \in k$, the elements in $\m_D$ are in infinitely many 
height 1 primes of $B$. It follows that $B$ is not a Krull domain and is not Noetherian. 

The center of $\beta_a$ on $B$ is the maximal ideal $(\frac{x + ay}{y}, \frac{y^2}{x + ay})B$. 
Since $\beta_a$ dominates $\alpha_a$ for each $a \in k$,  the element $\frac{x + ay}{y}$ is a 
unit in $\beta_b$ for each $b \in k$ with $b \ne a$.  
It follows that the representation $B = \bigcap_{a \in k}\beta_a$ is irredundant. 

The center on $B$  of the order valuation ring $\ord_D$ of $D$ is 
the ideal 
$$
\p = (x, y, \{\frac{y^2}{x + ay} ~|~ a \in k \})B.
$$
Then $\p$ is a nonmaximal
prime of $B$  and $B_\p = \ord_D$.

Since  $\beta_a = B_{(\frac{x+ay}{y}, \frac{y^2}{x+ay})B}$ for each $a \in k$, each 
 $\beta_a$ is essential in the representation $B = \bigcap_{a\in k}\beta_a$. 
\end{proof}

\begin{remark}
The prime divisors of the 2nd kind for   $D$ 
that are centered on height~1 primes of $C$  are the   order 
valuation rings of the $\alpha_a$ and  the order valuation ring  of  $\alpha'$.

Let $\p_a = (\frac{x + ay}{y})\beta_a$.  Then $\p_a$ is a height 1 prime of $\beta_a$ and 
 $(\beta_a)_{\p_a}$ is the order valuation ring of $\alpha_a$.
The image of $\frac{y^2}{x + ay}$  in the field of fractions  of $\beta_a/\p_a$ is transcendental over $k$. 
Hence  $\p_a \cap C$ is a height 1 prime ideal of $C$ and  $ C_{\p_a \cap C} = (\beta_a)_{\p_a}$.

Let $\p' = (\frac{y}{x})\beta'$.  Then $\p'$ is a height 1 prime of $\beta'$ and $\beta'_{\p'}$ is the 
order valuation ring of $\alpha'$.  
The image of $\frac{x^2}{y}$  in the field of fractions  of $\beta'/\p'$ is transcendental over $k$. 
Hence  $\p' \cap C$ is a height 1 prime ideal of $C$ and  $ C_{\p' \cap C} = (\beta')_{\p'}$.

The height 1 prime $\p$ of $B$ such that $B_\p = \ord_D$ has the property that $\p \cap C = \m_C$,
the maximal ideal of $C$.  If $\q$ is a height 1 prime of $B$ such that $\q \neq \p$,  
then $\q \cap C$ is a height 1 prime of $C$ and $C_{(\q \cap C)} = B_\q$.  
\end{remark}

Theorem~\ref{thm6.4} describes a  non-Noetherian local domain in $\R(D)$.

\begin{theorem}  \label{thm6.4}
Assume notation as in Setting~\ref{alanexample}  and in Theorem~\ref{thm6.5}.   
Then  $C$  is a 
non-Noetherian local domain in $\R(D)$.  Indeed, 
$C =  C' =   \bigcap_{a \in k}\beta_a  \cap \beta'   = B \cap \beta' $.  \end{theorem} 
\begin{proof} 
It is  established in Theorem~\ref{thm6.5}    that $C \subseteq B$. 
To see that $C \subseteq C'$,  we show that $\frac{y^2}{x + ay} \in \beta'$.
Consider $(x + ay)\alpha'$ and set $y_1 = \frac{y}{x}$.
Then $x + ay = x + axy_1 = x(1 + ay_1)$.  Since $1 + ay_1$ is a unit of $\alpha'$ 
and $y^2 \in x \alpha'$, it follows that $y^2 \in (x+ ay)\beta'$. 
Therefore,  for each $a \in k$,    $\frac{y^2}{x + ay} \in B  \cap \beta'  $.   

It is clear that $\frac{x^2}{y}   \in \beta'$,  and $\frac{x^2}{y} \in \beta_a \iff x^2 \in y\beta_a$.
Since   $y_1 = \frac{y}{x} \in \alpha_a $,   we have
$x^2 \in y\alpha_a = xy_1\alpha_a$  because $x \in y_1\alpha_a$.  Then $\alpha_a  \subseteq \beta_a$
implies that $C \subseteq \bigcap_{a \in k} \beta_a \cap \beta' $.

It remains to show that $C' = B  \cap \beta'     \subseteq C$.  To prove this,  it suffices to prove
that each minimal valuation overring  of $C$ contains $C'$.  A 
minimal valuation overring $V$  of $C$
dominates    $C$ by Remark~\ref{note2.1}.5.  Therefore  $V$   also dominates $D$. 

Let $V$ be a    minimal   valuation overring of $C$,  and let $v$ denote 
a valuation associated to $V$.
Then  $v(x) > 0$ and $v(y) > 0$.  
Since $x^2/y$ and 
$y^2/x$ are in the maximal ideal of $C$,  we also have $v(x^2) > v(y)$ and $v(y^2) > v(x)$.

If $v(x) > v(y)$,  then $D[x/y]_{(y,  x/y)D[x/y]}~ \subseteq  ~  V$.   Since $y^2/x \in C$,  it follows that 
$D[x/y, y^2/x]_{(x/y,  y^2/x)D[x/y, y^2/x]}~ \subseteq ~V$.  Therefore $C' \subseteq V$ in this case.

If $v(y) > v(x)$,  then $D[y/x]_{(x,  y/x)D[y/x]}~ \subseteq  ~  V$.   Since $x^2/y \in C$,  it follows that 
$D[y/x, x^2/y]_{(y/x,  x^2/y)D[y/x, x^2/y]}~ \subseteq ~V$.  Therefore $C' \subseteq V$ in this case.

If $v(x) = v(y)$,   then $D[y/x] \subset V$.  Since  $D[y/x]_{xD[y/x]} = \ord_D$ and $\ord_D$ contains $ C'$,  
we may assume that $V \neq \ord_D$.  Then $V$ is centered on a maximal ideal of $D[y/x]$ that lies over $\m_D$. 
Since $D$ has an algebraically closed coefficient field $k$, it follows that there exists $a \in k$ such that 
$v(x + ay) > v(x) = v(y)$.  Therefore 
$$
\alpha_a  ~=  ~ D[\frac{x + ay}{y}]_{(y, \frac{x+ ay}{y})}  ~ \subseteq  ~V
  $$

and $V$ dominates $\alpha_a$.  Then $y^2/(x + ay)  ~\in ~  \m_C  ~\subseteq  ~\m_V$ implies that 
$$  
  \beta_a ~ =~ \alpha_a[\frac{y^2}{x + ay}]_{(\frac{x + ay}{y}, \frac{y^2}{x + ay})} ~ \subseteq   ~ V.
  $$

Therefore $C'  \subseteq V$ also in this case. This completes the proof that $C = C'$.    
        \end{proof} 

\begin{remark} \label{rm6.15}
{The infinitely near points $\beta_a$ and $\beta'$ of Setting~\ref{alanexample} are all proximate
to $D$.  Hence the ring $B$ of Theorem~\ref{thm6.5} and the ring
$C$ of Theorem~\ref{thm6.4} are defined by Noetherian subsets of $Q(D)$.} 
\end{remark} 

The non-Noetherian rings $R \in \R(D)$ that we have described have 
infinitely many height one primes that contain  $\m_D$.  This motivates us to ask: 

\begin{question} \label{ques6.19}
Let  $R \in \R(D)$.   If $\Spec R$ is  a Noetherian topological space,
does it follow that $R$  is a Noetherian ring?
\end{question} 

\begin{comments}
Let $R = \O_\U$,   where  $\U \subset Q(D)$ is  as 
in Remark~\ref{rm3.3},  that is,  the rings $\alpha \in \U$ are incomparable and are minimal in $Q(D)$ 
with respect to containing $R$. 

 Let 
$\E = \{ V ~|~ V \text{ is an essential valuation ring for some } \alpha \in \U \}.$    It is clear that 
$R = \bigcap\{ V ~|~ V \in \E \}$.  
Each $V \in \E$ is 
either an essential valuation for $D$ or a prime divisor of the second kind on $D$.   If $V$ is 
a prime divisor of the second kind on $D$, then $\m_D$ is contained in the center of $V$ on $R$. 
Related to Question~\ref{ques6.19},  we have:

\begin{enumerate}
\item
There exist Noetherian subspaces $\U$ of $Q(D)$ for which the ring $R = \O_\U$ 
fails to have Noetherian spectrum.  The ring $B$ in Theorem~\ref{thm6.5} and 
the ring $C$ in Theorem~\ref{thm6.4} both fail to have Noetherian spectrum,  and both 
are defined by Noetherian subspaces of $Q(D)$.

\item    If 
there are only finitely many $V \in \E$ that are of the second kind on $D$, then the set $\E$ has
finite character and $R$ is a Krull domain, so $R$ is Noetherian. 

\item
If $\Spec R$ is Noetherian, then $\m_D$ is contained in only finitely many height one primes of $R$. 

\item
Each $W \in \E$ that is irredundant in the representation $R = \bigcap \{V ~| ~V \in \E \}$ is a
localization of $R$ and thus is centered on a height one prime of $R$, cf. \cite[Theorem 5.1, p. 330]{Ohm}.
Therefore $\Spec R$ 
is Noetherian  implies only finitely many $W \in \E$ of the second kind on $D$ are irredundant 
in the representation  $R = \bigcap \{V ~| ~V \in \E \}$.

\end{enumerate}

\end{comments}

\begin{remark}   \label{rmk6.21}
As a generalization of the situation  considered in this paper and in \cite{HO},  let
$A$ be a 2-dimensional Noetherian regular  integral domain\footnote{For example, $A$ could be a 
polynomial ring in 2 variables over a field.}  and let $Q(A)$  denote the 
``quadratic tree'' of   2-dimensional regular local overrings of $A$.  The set $Q(A)$ is 
partially ordered with respect to inclusion and has properties similar to  those used 
in this paper and in \cite{HO}.  

Let $\R(A)$ denote the family of rings obtained as intersections of rings in $Q(A)$.
It would be interesting to examine topological properties of $Q(A)$, and to 
examine the structure of rings in the set $\R(A)$.

\end{remark}




\end{document}